\documentclass[reqno, 10pt]{amsart}
\usepackage{amssymb}
\usepackage{amsmath}
\usepackage[mathscr]{euscript}
\usepackage{hyperref}
\usepackage{graphicx}
\usepackage[normalem]{ulem}
\usepackage{cancel}

\usepackage{amssymb}
\usepackage{hyperref}
\RequirePackage[dvipsnames]{xcolor} 
\definecolor{halfgray}{gray}{0.55} 
\definecolor{webgreen}{rgb}{0,0.5,0}
\definecolor{webbrown}{rgb}{.6,0,0} \hypersetup{%
  colorlinks=true, linktocpage=true, pdfstartpage=3,
  pdfstartview=FitV,%
  breaklinks=true, pdfpagemode=UseNone, pageanchor=true,
  pdfpagemode=UseOutlines,%
  plainpages=false, bookmarksnumbered, bookmarksopen=true,
  bookmarksopenlevel=1,%
  hypertexnames=true,
  pdfhighlight=/O,
  urlcolor=webbrown, linkcolor=RoyalBlue,
  citecolor=webgreen, 
  pdftitle={},%
  pdfauthor={},%
  pdfsubject={2000 MAthematical Subject Classification: Primary:},%
  pdfkeywords={},%
  pdfcreator={pdfLaTeX},%
  pdfproducer={LaTeX with hyperref}%
}

\usepackage{enumitem}

\newtheorem{theorem}{Theorem}[section]
\newtheorem{lemma}[theorem]{Lemma}
\newtheorem{corollary}[theorem]{Corollary}
\newtheorem{proposition}[theorem]{Proposition}
\newtheorem{claim}[theorem]{Claim}

\theoremstyle{definition}
\newtheorem{definition}[theorem]{Definition}
\newtheorem{remark}[theorem]{Remark}

\def\R{\mathbb{R}}
\def\N{\mathbb{N}}

\def\sm{\setminus}
\renewcommand{\phi}{\varphi}
\def\loc{\mathrm{loc}}
\def\Gl{{GL}}
\def\Sl{{SL}}

\usepackage{mdframed}

\def\P{{\mathbb{P}}}
\def\Pone{{\mathbb{P}^{1}}}

\newcommand{\restrict}[2]{{#1}{|_{{ #2}}}}
\newcommand{\inv}{^{-1}}
\renewcommand{\setminus}{\smallsetminus}



\begin{document}

\title
[Continuity of Lyapunov Exponents]
{Continuity of Lyapunov Exponents for Cocycles with Invariant Holonomies}

\author[L.~Backes]{Lucas Backes} \address{\noindent IME - Universidade do Estado do Rio de Janeiro, Rua S\~ao Francisco Xavier 524, CEP 20550-900, Rio de Janeiro, RJ, Brazil . 
\newline e-mail: \rm
  \texttt{lhbackes@impa.br} }

\author[A.~Brown]{Aaron Brown}\address{\noindent  Department of Mathematics, University of Chicago,  5734 S University Ave, Chicago, IL 60637
\newline e-mail: \rm
  \texttt{awbrown@math.uchicago.edu} }

\author[C.~Butler]{Clark Butler}\address{\noindent Department of Mathematics, University of Chicago,  5734 S University Ave, Chicago, IL 60637
  \newline e-mail: \rm
  \texttt{cbutler@math.uchicago.edu}}

\thanks{The third author was supported by the National Science Foundation Graduate Research Fellowship under Grant No. DGE-1144082.}

  
\begin{abstract}
We prove a conjecture of Viana which states that Lyapunov exponents vary continuously when restricted to $GL(2,\mathbb{R})$-valued cocycles over a subshift of finite type which admit invariant holonomies that depend continuously on the cocycle. 
\end{abstract}

\maketitle

\section{Introduction}
Consider an invertible measure preserving transformation $f\colon (X, \mu)\to (X,\mu)$ of a standard probability space. For simplicity, assume $\mu$ to be ergodic. 
Given  a measurable function $A\colon X\to GL(d, \mathbb{R})$ we define the linear cocycle over $f$ by the dynamically defined products
\begin{equation}\label{def:cocycles}
A^n(x)=
\left\{
	\begin{array}{ll}
		A(f^{n-1}(x))\ldots A(f(x))A(x)  & \mbox{if } n>0 \\
		Id & \mbox{if } n=0 \\
		(A^{-n}(f^{n}(x)))^{-1}=A(f^{n}(x))^{-1}\ldots A(f^{-1}(x))^{-1}& \mbox{if } n<0. \\
	\end{array}
\right.
\end{equation}

Under certain integrability hypotheses (for instance if the range of $A$ is bounded),  Oseledets theorem guarantees the existence of numbers $\lambda _1>\ldots > \lambda _{k}$, called the \emph{Lyapunov exponents}, and a decomposition $\mathbb{R}^d=E^1_{x}\oplus \ldots \oplus E^k_{x}$, called the \emph{Oseledets splitting}, into vector subspaces depending measurably on $x$ such that for almost every $x$
\begin{displaymath}
A(x)E^i_{x}=E^i_{f(x)} \; \textrm{and} \; \lambda _i =\lim _{n\to \pm  \infty} \dfrac{1}{n}\log \| A^n(x)v\| 
\end{displaymath}
for every non-zero $v\in E^i_{x}$ and $1\leq i \leq k$. 

Lyapunov exponents arrive naturally in the study smooth dynamics. Indeed, given a diffeomorphism of a manifold that preserves a probability measure, the derivative determines a natural cocycle associated to the  system.      The corresponding Lyapunov exponents  play a central role in the modern study of dynamical systems. For instance, given a $C^2$ diffeomorphism preserving a measure with negative exponents, Pesin constructed stable  manifolds through almost every point \cite{MR0458490}.  
 Moreover, Lyapunov exponents are deeply connected with the entropy of smooth dynamical systems and the geometry of measure as shown by the entropy formulas of Ruelle \cite{MR516310}, Pesin \cite{MR0466791}, and Ledrappier--Young \cite{MR819556,MR819557}.

In the present paper, we are interested in the continuity properties of Lyapunov exponents as one varies the cocycle and the underlying measure while keeping the base dynamics constant.  Our base dyanmics will be a subshift of finite type or, more generally, a hyperbolic set and our measures will always be taken to   be measures admitting a local product structure.  
As a corollary of our main result, we obtain continuity of Lyapunov exponents for fiber-bunched cocycles in the space of H\"older  continuous cocycles, 
giving  an affirmative answer to a conjecture   \cite[Conjecture 10.12]{V2} of Viana 
(see Sections  \ref{sec:definitions} and \ref{sec: cor}  
for precise definitions and statements):

\begin{theorem} \label{conjecture}
Lyapunov exponents vary continuously restricted to the subset of fiber-bunched elements $A\colon M\rightarrow GL(2,\mathbb{R})$ of the space $H^r(M)$. 
\end{theorem}

In general, one can not expect to obtain continuity of Lyapunov exponents in the space of H\"older cocycles without any extra assumption. Indeed, in \cite{BocV}, Bocker and Viana presented an example of a H\"older-continuous, $SL(2,\mathbb{R})$-valued  cocycle with non-zero Lyapunov exponents which is approximated in the H\"older topology by cocycles with zero Lyapunov exponents. Recently the third author \cite{But} has refined the Bocker-Viana construction to build a family of examples of discontinuity of Lyapunov exponents in the H\"older topology which are arbitrarily close to being fiber-bunched. Theorem \ref{conjecture} is sharp for this family. 

The technique employed by Bocker and Viana to construct their example is a refinement of a technique used by Bochi \cite{Bochi, Bochi2} to prove the Bochi--Ma\~n\'e theorem. This theorem implies that, in the space of continuous cocycles over aperiodic  base dynamics, the only continuity points for Lyapunov exponents of $SL(2,\mathbb{R})$-valued  cocycles are those which are (uniformly) hyperbolic and those with zero exponents.  Thus, discontinuity of Lyapunov exponents is typical if one only assumes continuous variation of the cocycle.

The main dynamical feature exhibited by fiber-bunched cocycles is the existence of a continuous family of invariant holonomies. These holonomies moreover vary continuously with the cocycle.  
This is the main geometric property we exploit to establish the continuity of Lyapunov exponents. Our main theorem below states that Lyapunov exponents depend continuously on the cocycle and on the underlying measure if we restrict ourselves to families of cocycles admitting invariant holonomies and to families of invariant measures with local product structure and ``well behaved" Jacobians. 

Even though discontinuity of Lyapunov exponents is a quite common feature as we pointed out above, there are some contexts where continuity has been previously established. For instance, Furstenberg and Kifer \cite{FK, Kif} established continuity of the largest Lyapunov exponent for i.i.d.\ random matrices under certain irreducibility conditions. In the same setting, but under assumption of strong irreducibility and a certain contraction property, Le Page \cite{LePage1, LePage1} showed local H\"older continuity and even smoothness of Lyapunov exponents. Duarte and Klein \cite{DK} derived H\"older continuity of the Lyapunov exponents for a class irreducible Markov cocycles. 
In certain cases one can obtain real-analyticity of the Lyapunov exponents \cite{Rue, Pe}. Continuity has also proven in the context of Schr\"odinger cocycles by Bourgain and Jitomirskaya \cite{Bou, BJ}. More recently, Bocker and Viana \cite{BocV} and Malheiro and Viana \cite{1410.1411} proved continuity of Lyapunov exponents for random products of 2-dimensional matrices in the Bernoulli and Markov settings. Our result extends the results of \cite{BocV} and  \cite{1410.1411}.  
In higher dimensions, continuous dependence of all Lyapunov exponents for i.i.d.\ random products of matrices in $GL(d,\mathbb{R})$ was announced by Avila, Eskin, and Viana \cite{AEV}.

\section{Definitions and statement of main theorem}\label{sec:definitions}
\subsection{Subshifts of finite type}
Let $Q = (q_{ij})_{1 \leq i,j\leq \ell}$ be an $\ell \times \ell$ matrix with $q_{ij} \in \{0,1\}$. The \emph{subshift of finite type} associated to the matrix $Q$ is the subset of the bi-infinite sequences $\{1,\dots, \ell\}^{\mathbb{Z}}$ satisfying 
\[
\hat{\Sigma} = \{(x_{n})_{n \in \mathbb{Z}}: q_{x_{n}x_{n+1}} = 1 \;\text{for all $n \in \mathbb{Z}$}\}. 
\] 
We require that each row and column of $Q$ contains at least one nonzero entry. We let $\hat{f}: \hat{\Sigma} \rightarrow \hat{\Sigma}$ be the left-shift map defined by $\hat{f}(x_{n})_{n \in \mathbb{Z}} = (x_{n+1})_{n \in \mathbb{Z}}$.  We will always assume that $\hat{f}$ is topologically transitive on $\hat{\Sigma}$. We let
\[
\Sigma^{u} = \{(x_{n})_{n \geq 0}: q_{x_{n}x_{n+1}} = 1 \;\text{for all $n \geq 0$}\},
\]
\[
\Sigma^{s} = \{(x_{n})_{n \leq 0}: q_{x_{n}x_{n+1}} = 1 \;\text{for all $n \leq -1$}\}.
\] 
We have projections $P^{u}: \hat{\Sigma} \rightarrow \Sigma^{u}$ and $P^{s}: \hat{\Sigma} \rightarrow \Sigma^{s}$ obtained by dropping all of the negative coordinates and all of the positive coordinates, respectively, of a sequence in $\hat{\Sigma}$. We let $f_{s}$ and $f_{u}$ denote the right and left shifts on $\Sigma^{s}$ and $\Sigma^{u}$, respectively. 

We define the \emph{local stable set} of $\hat{x} \in \hat{\Sigma}$ to be 
\[
W^{s}_{loc}(\hat{x}) = \{(y_{n})_{n \in\mathbb{Z}} \in \hat{\Sigma}:x_{n}  = y_{n} \; \text{for all $n \geq 0$}\},
\] 
and the \emph{local unstable set} to be 
\[
W^{u}_{loc}(\hat{x}) = \{(y_{n})_{n \in \mathbb{Z}} \in \hat{\Sigma}:x_{n}  = y_{n} \; \text{for all $n \leq 0$}\}.
\]
We think of $\Sigma^{s}$ and $\Sigma^{u}$, respectively, as parametrizations of the local stable and unstable sets. We define 
\[
\Omega^{s} = \{(\hat{x},\hat{y}) \in \hat{\Sigma} \times \hat{\Sigma}: \hat{y} \in W^{s}_{loc}(\hat{x})\},
\]
\[
\Omega^{u} = \{(\hat{x},\hat{y}) \in \hat{\Sigma} \times \hat{\Sigma}: \hat{y} \in W^{u}_{loc}(\hat{x})\}.
\]
$\Omega^{s}$ and $\Omega^{u}$ can be expressed locally as the product of a cylinder in $\hat{\Sigma}$ with $\Sigma^{s}$ and $\Sigma^{u}$, respectively. For $x \in \Sigma^{u}$ we define $W^{s}_{loc}(x) = (P^{u})^{-1}(x)$ and for $y \in \Sigma^{s}$ we define $W^{u}_{loc}(y) = (P^{s})^{-1}(y)$. Observe that if $\hat{x} \in \hat{\Sigma}$, then $W^{s}_{loc}(\hat{x}) = W^{s}_{loc}(P^{u}(\hat{x}))$. 

Each $\theta \in (0,1)$ gives rise to a metric on $\hat{\Sigma}$,
\begin{displaymath}
d_{\theta}(\hat{x},\hat{y})=\theta ^{N(\hat{x},\hat{y})},\; \textrm{where} \; N(\hat{x},\hat{y})=\max \lbrace N\geq 0; x_n=y_n \; \textrm{for all} \mid n\mid <N \rbrace.
\end{displaymath}
These metrics are all H\"older equivalent to one another and thus each defines the same topology on $\hat{\Sigma}$. 

For $m \in \mathbb{Z}$ and $a_{0},\dots,a_{k} \in \{1,\dots,\ell\}$, we define the cylinder notation
\[
[m;a_{0},\dots,a_{k}] = \{\hat{x} \in \hat{\Sigma}: x_{m+i} = a_{m+i}, \, 0 \leq i \leq k\}.
\]

\subsection{Stable and unstable holonomies}\label{sec: holonomies}
A $d$-dimensional \emph{linear cocycle} $\hat{A}$ over $\hat{f}$ is a map $\hat{A}: \hat{\Sigma} \rightarrow GL(d,\mathbb{R})$. 

\begin{definition}\label{defn: invariant holonomies}
A \emph{stable holonomy} for a linear cocycle $\hat{A}$ over $\hat{f}$ is a collection of linear maps $H^{s,\hat{A}}_{\hat{x} \hat{y}} \in GL(d,\mathbb{R})$ defined for $\hat{y} \in W^{s}_{loc}(\hat{x})$ which satisfy the following properties, 
\begin{itemize}
\item 
$H^{s,\hat{A}}_{\hat{y} \hat{z}}=H^{s,\hat{A}}_{\hat{x}\hat{z}}H^{s,\hat{A}}_{\hat{y}\hat{x}} \quad \textrm{and} \quad H^{s,\hat{A}}_{\hat{x} \hat{x}} = Id$;

\item 
$H^{s,\hat{A}}_{\hat{f}(\hat{y})\hat{f}(\hat{z})}=\hat{A}(\hat{z})H^{s,\hat{A}}_{\hat{y}\hat{z}}\hat{A}(\hat{y})^{-1};$

\item
The map $\Omega^{s} \times \mathbb{R}^{d} \rightarrow GL(d,\mathbb{R})$ given by $(\hat{x},\hat{y},v) \rightarrow H^{s,\hat{A}}_{\hat{x}\hat{y}}(v)$ is continuous. 

\end{itemize}

\end{definition}
By replacing $\hat{f}$ and $\hat{A}$ with the inverse cocycle $\hat{A}^{-1}$ over $\hat{f}^{-1}$, we get an analogous definition of unstable holonomies $H^{u,\hat{A}}_{\hat{x}\hat{y}}$ for $\hat{y} \in W^{u}_{loc}(\hat{x})$.

Stable and unstable holonomies for linear cocycles are not unique in general, even if the cocycle is locally constant (see \cite{KS}). To circumvent this issue we define a \emph{cocycle with holonomies} to be a triple $(\hat{A},H^{s,\hat{A}},H^{u,\hat{A}})$ where $\hat{A}$ is a linear cocycle over $\hat{f}$ and $H^{s,\hat{A}}$ and $H^{u,\hat{A}}$ are a stable and unstable holonomy for $\hat{A}$, respectively. We let $\mathcal{H}$ denote the space of all cocycles with holonomies, endowed with the subspace topology given by the inclusion
\[
\mathcal{H} \hookrightarrow C^{0}(\hat{\Sigma},GL(d,\mathbb{R})) \times C^{0}(\Omega^{s},GL(d,\mathbb{R})) \times C^{0}(\Omega^{u},GL(d,\mathbb{R})),
\]
where $\mathcal{H}$ is cut out by the linear equations in Definition \ref{defn: invariant holonomies} and these spaces of maps have the uniform topology. This means that a sequence of cocycles with holonomies $\{(\hat{A}_{n},H^{s,\hat{A}_{n}},H^{u,\hat{A}_{n}})\}_{n \in \mathbb{N}}$ converges to $(\hat{A},H^{s,\hat{A}},H^{u,\hat{A}})$ if $\hat{A}_{n} \rightarrow \hat{A}$ uniformly and the stable and unstable holonomies converge uniformly on local stable and unstable sets, respectively. 

\begin{definition}
A sequence of linear cocycles $\{\hat{A}_{n}\}_{n \in \mathbb{N}}$ over $\hat{f}$ \emph{converges uniformly with holonomies} to a linear cocycle $\hat{A}$ if for each $n$ there is a triple $(\hat{A}_{n},H^{s,\hat{A}_{n}},H^{u,\hat{A}_{n}}) \in \mathcal{H}$ such that this sequence converges in $\mathcal{H}$ to a triple $(\hat{A},H^{s,\hat{A}},H^{u,\hat{A}})$ defining a stable and unstable holonomy for $\hat{A}$. 
\end{definition}

\begin{remark}
If $\hat{A}$ is $\alpha$-H\"older continuous and $\alpha$-fiber-bunched (see Definition \ref{defn: fiber bunched}) or locally constant, there is a canonical stable holonomy for $\hat{A}$ defined by the formula
\begin{displaymath}
H^{s,\hat{A}}_{\hat{x}\hat{y}} = \lim _{n\rightarrow \infty}\hat{A}^n(\hat{y})^{-1}\hat{A}^n(\hat{x}), \; \; \hat{y} \in W^{s}_{loc}(\hat{x}).
\end{displaymath}
Our definition of stable and unstable holonomies is more general and does not imply that the sequence on the right converges. 
\end{remark}

\subsection{Product structure of measures}\label{sec: prod} 
For an $\hat{f}$-invariant measure $\hat{\mu}$ on $\hat \Sigma$ we let $\mu^{u} = P^{u}_{*}\hat{\mu}$ and $\mu^{s}= P^{s}_{*}\hat{\mu}$. The map $[0;i] \rightarrow P^{s}([0;i]) \times P^{u}([0;i]) $ induced by $\hat{x} \rightarrow (P^{s}(\hat{x}),P^{u}(\hat{x}))$ is a homeomorphism. We say that an $\hat{f}$-invariant measure $\hat{\mu}$ on $\hat{\Sigma}$ has \emph{local product structure} if there is a positive continuous function $\psi: \hat{\Sigma} \rightarrow (0,\infty)$ such that the restriction is of the form 
\[
\hat{\mu}|_{[0;i]} = \psi (\mu^{s}|_{P^{s}([0;i]) } \times \mu^{u}|_{P^{u}([0;i] )}).
\]

A \emph{Jacobian} of the  measure $\mu^u$ \emph{with respect to the dynamics $f_u$} is  the measurable function $J_{\mu^u} f_u$ such that $${d \left((f_u)_* (\mu^u |_{[0;i]} ) \right)(f(y)) = (J_{\mu^u} f_u(y))^{-1} \  d \mu^u(f(y)).}$$ 
A Jacobian of $\mu^{s}$ with respect to $f_{s}$ is defined similarly. 

Lemmas \ref{lemma: Jacobians} and \ref{lemma:continuity mu} below give consequences of the existence of local product structure for $\mu$ which are well known, see for instance \cite[Lemmas 2.1, 2.2]{BV}. We reproduce the proofs here to indicate explicitly how the local product structure of $\hat{\mu}$ is used; in particular, we emphasize in the proofs that the Jacobians and disintegrations constructed depend continuously on the function $\psi$ which gives the local product structure of $\hat{\mu}$ via explicit formulas. 

\begin{lemma}\label{lemma: Jacobians}
Assume $\hat \mu$ has local product structure.
Then the measure $\mu ^u$ admits a continuous positive Jacobian $J_{\mu ^u}f_{u}$ with respect to the map $f_u$. Similarly $\mu ^{s}$ admits a continuous positive Jacobian $J_{\mu ^s}f_{s}$ with respect to the map $f_{s}$. 
\end{lemma}

\begin{proof}
Let $y\in \Sigma^{u}$ and $D$ be any measurable set containing $y$ and contained in a cylinder $[0;i,j]$. Thus, by definition,
\begin{displaymath}
\mu^{u} (f_{u}(D))=\hat{\mu}((P^u)^{-1}(f_{u}(D)))=\int _{\lbrace x\in f_{u}(D)\rbrace }\psi (x,z)d\mu^u (x)d\mu ^s(z) 
\end{displaymath}
and moreover,
\begin{displaymath}
\mu^u(D)=\hat{\mu}((P^u)^{-1}(D))=\hat{\mu}(\hat{f}((P^u)^{-1}(D)))=\int _{\lbrace x\in f_{u}(D), z_{-1}=i \rbrace }\psi (x,z)d\mu^u (x)d\mu ^s(z)
\end{displaymath}
where in the second equality we have used the $\hat{f}$-invariance of $\hat{\mu}$. Now, letting $D$ shrink to $\{y\}$ we get that
\begin{displaymath}
\dfrac{\mu^u(f_{u}(D))}{\mu^u(D)}\rightarrow \dfrac{1}{\int _{\lbrace z_{-1}=i \rbrace }\psi (f_{u}(y),z)d\mu ^s(z)}.
\end{displaymath}
Defining $J_{\mu ^u}f_{u}(y):=\dfrac{1}{\int _{\lbrace z_{-1}=i \rbrace }\psi (f_{u}(y),z)d\mu ^s(z)}$, which is clearly positive and continuous, we get the desired result. An analogous proof replacing $f_{u}$ by $f_{s}$ shows that $f_{s}$ admits a continuous positive Jacobian $J_{\mu ^s}f_{s}$ with respect to $\mu^{s}$. 
\end{proof}

Given $x, y \in \Sigma^{u}$ in the same cylinder $P^{u}([0;i])$, we define the \textit{unstable holonomy map} $h_{x,y}:W^s_{\loc}(x)\rightarrow W^s_{\loc}(y)$, by assigning to each $\hat{x}\in W^s_{\loc}(x)$ the unique $\hat{y}=h_{x,y}(\hat{x})\in W^s_{\loc}(y)$ with $\hat{y} \in W^{u}_{\loc}(\hat{x})$. 

The partition of $(\Sigma,\hat \mu)$ into local stable manifolds is a measurable partition and thus 
induces a disintegration into a family of  conditional measures $\{\hat{\mu}_x\}_{x\in \Sigma^{u}}$ with each $\hat \mu_x$ supported on $W^s_{\loc}(x)$.  All such families agree up to null sets. 

Using the local product structure of the measure $\hat{\mu}$ we have the following.  

\begin{lemma}\label{lemma:continuity mu} Assume $\hat \mu$ has local product structure.  
Then the measure $\hat{\mu}$ has a disintegration into conditional measures $\{\hat{\mu}_x\}_{x\in \Sigma^{u}}$ that vary continuously with $x$ in the weak-$*$ topology. In fact, for every $x, y \in \Sigma^{u}$ in the same cylinder $[0;i]$,
\begin{displaymath}
h_{x,y}:(W^s_{\loc}(x),\hat{\mu}_x))\rightarrow (W^s_{\loc}(y),\hat{\mu}_y)
\end{displaymath} 
is absolutely continuous, with Jacobian $R_{x,y}$ depending continuously on $(x,y)$. 
\end{lemma}

\begin{proof}
For each $i \in \{1,\dots,\ell\}$, the local product structure of $\hat{\mu}$ allows us to express $\hat{\mu}|_{[0;i]}$ as $\psi \cdot (\mu^{s}|_{P^{s}([0;i])} \times \mu^{u}|_{P^{u}([0;i])})$ for a positive continuous function $\psi$. Since $\mu ^u=P^u_{\ast}\hat{\mu}$ we have that $\int _{W^s_{loc}(x)}\psi (\hat{x})d\mu ^s(\hat{x})=1$ on every local stable manifold and thus $\hat{\mu}_x=\psi (\hat{x}) \mu ^s$ and $R_{x,y}(\hat{x})=\dfrac{\psi (h_{x,y}(\hat{x}))}{\psi (\hat{x})}$ define a disintegration of $\hat{\mu}$ and a Jacobian for $h_{x,y}$ as we want. 
\end{proof}

\begin{remark}\label{remark: property of mu} 
Observe that, with the above disintegration of $\hat{\mu}$ and the Jacobians given in Lemma \ref{lemma: Jacobians}, we  have that 
\begin{displaymath}
\hat{\mu}_{f_u^n(y)}\mid _{\hat{f}^n(W^s_{\loc}(y))}=\dfrac{1}{J_{\mu ^u}f_u^n(y)}\hat{f}^n_{\ast}\hat{\mu}_y
\end{displaymath}
for every $y\in \Sigma ^u$.
\end{remark}

In order to state the Main theorem we need to formulate a notion of convergence of probability measures on $\hat{\Sigma}$ which is stronger than weak-$*$-convergence. We say that a sequence of $\hat{f}$-invariant probability measures $\{\hat{\mu}_{k}\}_{k \in \mathbb{N}}$ with local product structure converges to an $\hat{f}$-invariant measure $\hat{\mu}$ with local product structure if $\hat{\mu}_{k}$ converges to $\hat{\mu}$ in the weak-$*$ topology on probability measures on $\hat{\Sigma}$ and the positive continuous functions $\psi_{k}$ defining the local product structure of $\hat{\mu}_{k}$ converge uniformly to the function $\psi$ defining the local product structure of $\hat{\mu}$. Uniform convergence of $\psi_{k}$ to $\psi$ together with the weak-* convergence of $\hat{\mu}_{k}$ implies that the sequences of stable and unstable Jacobians $\{J_{\mu_{k}^{u}}f_{u}\}_{k \in \mathbb{N}}$ and $\{J_{\mu_{k}^{s}}f_{s}\}_{k \in \mathbb{N}}$ converge uniformly to $J_{\mu^{u}}f_{u}$ and $J_{\mu^{s}}f_{s}$, respectively, and that the conditional measures $\hat{\mu}_{x}^{k}$ of $\hat{\mu}_{k}$ along $\Sigma^{u}$ converge uniformly to the disintegration of $\hat{\mu}$.

As a shorthand for this notion of convergence we will say that ``$\hat{\mu}_{k}$ converges to $\hat{\mu}$ as in Section \ref{sec: prod}". A useful criterion for checking this notion of convergence as well as the existence of local product structure is given in the next lemma. 

\begin{lemma}\label{lemma: measure convergence}
Let $\hat{\mu}$ be an ergodic, fully supported probability measure on $\hat{\Sigma}$. Suppose that the projected measure $\mu^{u} = P^{u}_{*}\hat{\mu}$ admits a positive $\beta$-H\"older continuous Jacobian $ J_{\mu^{u}}f_{u}$ in the $\beta$-H\"older norm with respect to $f_{u}$. Then $\hat{\mu}$ has local product structure given by a function $\psi: \hat{\Sigma} \rightarrow (0,\infty)$ and moreover $\psi$ depends continuously on $ J_{\mu^{u}}f_{u}$ in the $\beta$-H\"older norm. 
\end{lemma}

\begin{proof}
The assertion that $\hat{\mu}$ admits local product structure follows from \cite[Lemmas 2.4, 2.6]{BV} since the Jacobian $J_{\mu^{u}}f_{u}$ is assumed to be H\"older continuous. To establish that $\psi$ depends continuously on  $ J_{\mu^{u}}f_{u}$, we recall the formula for $\psi$ derived in the course of the proof. 

Fix points $z_{i} \in P^{u}([0;i])$ for $1 \leq i \leq \l$. The construction of the local product structure in \cite{BV} gives the following formula for $\psi$: for $\hat{x} \in [0;i]$, 
\[
\psi(\hat{x}) = \lim_{n \rightarrow \infty} \frac{J_{\mu^{u}}f_{u}(P^{u}(\hat{f}^{n}(h_{x,z_{i}}(\hat{x}))))}{J_{\mu^{u}}f_{u}(P^{u}(\hat{f}^{n}(\hat{x})))}
\]
if we identify $\mu^{u}$ with the measure $\hat{\mu}_{z_{i}}$ from Lemma \ref{lemma:continuity mu}. A standard argument using distortion estimates shows that the limit on the right side exists and depends continuously on the function $J_{\mu^{u}}f_{u}$ in the $\beta$-H\"older topology, see \cite[Lemma 2.4]{BV} or the arguments at the beginning of the proof of Lemma \ref{lemma: equilibrium states}.
\end{proof}

As a consequence, if $\hat{\mu}_{k} \rightarrow \hat{\mu}$ is a sequence of measures converging in the weak-* topology all of which are ergodic, fully supported, and have local product structure, and moreover the Jacobians $J_{\mu^{u}_{k}}f_{u}$ are $\beta$-H\"older continuous and converge in the $\beta$-H\"older norm to $J_{\mu^{u}}f_{u}$, then $\hat{\mu}_{k}$ converges to $\hat{\mu}$ as in Section \ref{sec: prod}.

\subsection{Main  theorem}

For a continuous cocycle $\hat{A}$ over $\hat{f}$ and an $\hat{f}$-invariant probability measure $\hat{\mu}$ on $\hat{\Sigma}$, it follows by the Kingman Sub-Additive Ergodic Theorem    (  \cite{K}) 
that 
\[
\lambda _+(\hat{A},\hat{x})= \lim _{n\rightarrow \infty}\dfrac{1}{n}\log \|\hat{A}^n(\hat{x})\|
\]
and
\[
\lambda _-(\hat{A},\hat{x})= \lim _{n\rightarrow \infty}\dfrac{1}{n}\log \| (\hat{A}^n(\hat{x}))^{-1}\|^{-1}
\]
are well-defined at $\hat{\mu}$ almost every point $\hat{x} \in \hat{\Sigma}$. These are the (extremal) \emph{Lyapunov exponents} of $\hat{A}$. These functions are $\hat{f}$-invariant and hence, if $\hat{\mu}$ is ergodic with respect to $\hat{f}$, these functions are constant $\hat{\mu}$-a.e. In this case,   we define $\lambda _+(\hat{A},\hat{\mu})$ and $\lambda _{-}(\hat{A},\hat{\mu})$ to be the $\hat{\mu}$-a.e.\ constant values of the extremal Lyapunov exponents.

The main theorem of the paper is a criterion for joint continuity of the Lyapunov exponents $\lambda _+(\hat{A},\hat{\mu})$ and $\lambda _{-}(\hat{A},\hat{\mu})$ in the cocycle $\hat{A}$ and the measure $\hat{\mu}$ in the case when $\hat{A}$ is 2-dimensional. 

\begin{theorem}\label{mainthm}
Let $\{\hat{A}_{n}\}_{n \in \mathbb{N}}$ be a sequence of 2-dimensional linear cocycles over $\hat{f}$ converging uniformly with holonomies to a cocycle $\hat{A}$ and $\{\hat{\mu}_{n}\}_{n \in \mathbb{N}}$ a sequence of fully supported, ergodic, $\hat{f}$-invariant probability measures converging as in Section \ref{sec: prod} to an ergodic, $\hat{f}$-invariant measure $\hat{\mu}$ with local product structure and full support. Then $\lambda_{+}(\hat{A}_{n},\hat{\mu}_{n}) \rightarrow \lambda_{+}(\hat{A},\hat{\mu})$ and $\lambda_{-}(\hat{A}_{n},\hat{\mu}_{n}) \rightarrow \lambda_{-}(\hat{A},\hat{\mu})$. 
\end{theorem}

Theorem \ref{mainthm} provides an affirmative answer to \cite[Conjecture 10.13]{V2}.  
The proof of Theorem \ref{mainthm} begins in Section \ref{sec: preliminary}. We collect some corollaries of Theorem \ref{mainthm} in Section \ref{sec: cor} below. 

\section{Corollaries}\label{sec: cor}
In this section we demonstrate how to apply Theorem \ref{mainthm} to prove continuity of the Lyapunov exponents for certain classes of 2-dimensional linear cocycles over hyperbolic systems. We fix a $\theta \in (0,1)$ and for $\alpha > 0$ we let $C^{\alpha}(\hat{\Sigma},GL(d,\mathbb{R}))$ be the space of $\alpha$-H\"older continuous linear cocycles over the shift with respect to the metric $d_{\theta}$ on $\hat{\Sigma}$. $C^{\alpha}(\hat{\Sigma},GL(d,\mathbb{R}))$ is a Banach space with the $\alpha$-H\"older norm
\[
\|\hat{A}\|_{\alpha} = \sup_{\hat{x} \in \hat{\Sigma}}\|\hat{A}(\hat{x})\| + \sup_{\hat{x} \neq \hat{y} \in \hat{\Sigma}}\frac{\|\hat{A}(\hat{x})-\hat{A}(\hat{y})\|}{d_{\theta}(\hat{x},\hat{y})^{\alpha}}.
\]
\begin{definition}\label{defn: fiber bunched}
A linear cocycle $\hat{A}: \hat{\Sigma} \rightarrow GL(d,\mathbb{R}))$ is \emph{$\alpha$-fiber-bunched} if $\hat{A} \in C^{\alpha}(\hat{\Sigma},GL(d,\mathbb{R}))$ and there is an $N > 0$ such that 
\[
\|\hat{A}^{N}(\hat{x})\| \cdot \|(\hat{A}^{N}(\hat{x}))^{-1}\|^{-1} \cdot \theta^{-N\alpha} < 1
\]
for every $\hat{x} \in \hat{\Sigma}$. 
\end{definition}
The set of $\alpha$-fiber-bunched cocycles is open in $C^{\alpha}(\hat{\Sigma},GL(d,\mathbb{R}))$. 

For each H\"older continuous potential $\varphi: \hat{\Sigma} \rightarrow \mathbb{R}$ we may associate a unique \emph{equilibrium state} $\hat{\mu}_{\varphi}$ which is an ergodic, fully supported probability measure on $\hat{\Sigma}$ with local product structure \cite{Bow2, Lep}. The following lemma shows that H\"older-convergence of potentials implies convergence of equilibrium states as in Section \ref{sec: prod}. 

\begin{lemma}\label{lemma: equilibrium states}
If $\varphi_{k} \rightarrow \varphi$ in $C^{\beta}(\hat{\Sigma},\mathbb{R})$ for some $\beta > 0$ then $\hat{\mu}_{\varphi_{k}}$ converges to $\hat{\mu}_{\varphi}$ as in Section \ref{sec: prod}. 
\end{lemma}

\begin{proof}
We recall some well-known facts about equilibrium states which can be found in \cite{Bow2}. We first note that it suffices to prove the claim when the functions $\varphi_{k}$ are constant on the local stable sets of $\hat{f}$. For $\hat{x},\hat{y} \in [0;i]$ for some $1 \leq i \leq \ell$ we let $h^s_{\hat{x},\hat{y}}$ denote the stable holonomy from $W^{u}_{loc}(\hat{x})$ to $W^{u}_{loc}(\hat{y})$ which assigns to each $\hat{z}\in W^u_{\loc}(\hat{x})$ the unique point $h^s_{\hat{x},\hat{y}}(\hat{z})\in W^s_{\loc}(\hat{z})\cap W^{u}_{\loc}(\hat{y})$. Now fix $\ell$ points $\hat{z}_{1},\dots,\hat{z}_{\ell}$ such that $\hat{z}_{i} \in [0;i]$. We then define for $\hat{x} \in [0;i]$, 
\[
\psi_{k}^{u}(\hat{x}) = \sum_{j = 0}^{\infty} \varphi_{k}(\hat{f}^{j}(\hat{x})) - \varphi_{k}(\hat{f}^{j}(h^s_{\hat{x},\hat{z}_{i}}))
\]
and define $\varphi^{u}_{k}(\hat{x}) =\varphi_{k}( h^s_{\hat{x},\hat{z}_{i}}(\hat{x}))$. These functions then satisfy the equation  
\[
\varphi^{u}_{k} = \varphi_{k} + \psi_{k}^{u} \circ \hat{f} - \psi_{k}^{u}
\]
which implies that $\varphi^{u}_{k}$ is cohomologous to $\varphi_{k}$. Furthermore $\varphi^{u}_{k}$ is constant on local stable sets and thus descends to a continuous function on $\Sigma^{u}$ for each $k$. Since cohomologous potentials define the same equilibrium state we get that $\hat{\mu}_{\varphi^{u}_{k}} = \hat{\mu}_{\varphi_{k}}$ for all $k$.  Each function $\psi_{k}^{u}$ is $\beta$-H\"older continuous and thus so is $\varphi_{k}^{u}$, and as $k \rightarrow \infty$ convergence of $\varphi_{k}$ to $\varphi$ in $C^{\beta}(\hat{\Sigma},\mathbb{R})$ implies convergence of $\psi_{k}^{u}$ to the corresponding function $\psi_{k}$ for $\varphi$ in $C^{\beta}(\hat{\Sigma},\mathbb{R})$, and thus $\varphi^{u}_{k}$ converges to $\varphi^{u}$ in $C^{\beta}(\hat{\Sigma},\mathbb{R})$.  

Thus we may assume that $\varphi_{k}$ and $\varphi$ are constant on local stable sets of $\hat{f}$, and hence they descend to H\"older continuous functions on $\Sigma^{u}$. Recall that the transfer operator $T_{\varphi}: C^{0}(\Sigma^{u},\mathbb{C}) \to C^{0}(\Sigma^{u},\mathbb{C})  $ associated to $\varphi$ on $\Sigma^{u}$ is defined on continuous functions $g: \Sigma^{u} \rightarrow \mathbb{C}$ by 
\[
T_{\varphi}g(x) = \sum_{y \in f_{u}^{-1}(x)} e^{\varphi(y)}g(y).
\]
By the Ruelle-Perron-Frobenius theorem, $T_{\varphi}$ acts with a spectral gap on the Banach space $C^{\beta}(\Sigma^{u},\mathbb{C})$\cite{Bow2}. Let $\nu_{\varphi}^{u}$ be the dominant eigenvector for the adjoint action of $T_{\varphi}$ on probability measures and $\zeta_{\varphi}^{u} \in C^{\beta}(\Sigma^{u},\mathbb{C})$ the strictly positive dominant eigenvector for $T_{\varphi}$ which satisfies $\int_{\Sigma^{u}}\zeta_{\varphi}^{u} \,d\nu_{\varphi} = 1$ and has eigenvalue $e^{P}$, where $P$ is the topological pressure of $\varphi$. Then $\mu_{\varphi}^{u}$ is given by $\mu_{\varphi}^{u} = \zeta_{\varphi}^{u} \nu_{\varphi}^{u}$ and the Jacobian of $f_{u}$ with respect to  $\mu_{\varphi}^{u}$ is $e^{P-\varphi}\frac{\zeta_{\varphi}}{\zeta_{\varphi} \circ f_{u}}$.  

Since $T_{\varphi}$ depends continuously on $\varphi \in C^{\beta}(\hat{\Sigma},\mathbb{R})$, we conclude from the spectral gap property that the dominant eigenvector $\zeta_{\varphi}^{u}$ and its eigenvalue $e^{P}$ depend continuously on $\varphi$ in the $\beta$-H\"older norm and similarly the dominant eigenvector $\nu_{\varphi}^{u}$ for the adjoint depends  continuously on $\varphi$ in the weak-$*$ topology on probability measures on $\hat{\Sigma}$. Consequently convergence of $\varphi_{k}$ to $\varphi$ implies weak-* convergence of $\mu_{\varphi_{k}}^{u}$ to $\mu_{\varphi}^{u}$ and convergence in the $\beta$-H\"older norm of $J_{\mu^{u}_{\varphi_{k}}}f_{u}$ to $J_{\mu^{u}_{\varphi}}f_{u}$. By Lemma \ref{lemma: measure convergence} and the subsequent remark, we conclude that $\hat{\mu}_{\varphi_{k}}$ converges to $\hat{\mu}$ as in Section \ref{sec: prod}. 
\end{proof}

\begin{corollary}\label{cor: fiber bunched continuity} 
For each $\alpha, \beta > 0$, the Lyapunov exponents
\[
\lambda_{\pm}: C^{\alpha}(\hat{\Sigma},GL(2,\mathbb{R})) \times C^{\beta}(\hat{\Sigma},\mathbb{R}) \rightarrow \mathbb{R}
\] 
\[
(\hat{A},\varphi) \rightarrow \lambda_{\pm}(\hat{A},\hat{\mu}_{\varphi})
\]
are continuous when restricted to  $\hat{A} \in C^{\alpha}(\hat{\Sigma},GL(2,\mathbb{R}))$ which are $\alpha$-fiber-bunched. 
\end{corollary}

\begin{proof}
For fiber-bunched cocycles stable and unstable holonomies exist and moreover they vary continuously with respect to the cocycle in the $\alpha$-H\"older topology (see \cite{BGV} and \cite{VianaAlmostAllCocyc}). Lemma \ref{lemma: equilibrium states} implies that if $\varphi_{k}$ converges to $\varphi$ in $C^{\beta}(\hat{\Sigma},\mathbb{R})$ then the corresponding equilibrium states $\hat{\mu}_{\varphi_{k}}$ converge to $\hat{\mu}_{\varphi}$ as in Section \ref{sec: prod}. These two statements together then imply the corollary by Theorem \ref{mainthm}. 
\end{proof}

Continuous dependence of holonomies in the space of $\alpha$-fiber-bunched cocycles may actually be shown under slightly weaker hypotheses than convergence in the H\"older topology. It suffices to assume that the linear cocycle $A$ is $\alpha$-fiber-bunched, $A_{k}$ converges to $A$ in the $C^{0}$ topology, and each $A_{k}$ and $A$ are $\alpha$-H\"older continuous with uniformly bounded H\"older constant.  We refer the interested reader again to \cite{BGV} and \cite{VianaAlmostAllCocyc} for further details. 

For our second application we give an example of how to use Markov partitions to prove continuity of the Lyapunov exponents for cocycles over other hyperbolic systems besides subshifts of finite type. Let $M$ be a closed Riemannian manifold. Let $f:M \rightarrow M$ be an Anosov diffeomorphism, meaning that there is a $Df$-invariant splitting $TM = E^{s} \oplus E^{u}$ and constants $C > 0$, $0 < \nu < 1$ such that
\[
\|Df^{n}|E^{s}\| \leq C\nu^{n}, \; \; \|Df^{-n}|E^{u}\| \leq C\nu^{n}, \; n \geq 1.
\]
For $0 < \alpha \leq 1$ we say that $f$ is a $C^{1+\alpha}$ diffeomorphism of $M$ if $Df$ is $\alpha$-H\"older continuous. We write $\mathrm{Diff}^{1+\alpha}(M)$ for the space of $C^{1+\alpha}$ diffeomorphisms of $M$, equipped with the topology of uniform convergence for $f$ together with $\alpha$-H\"older convergence for the derivative $Df$. For a $C^{1+\alpha}$ Anosov diffeomorphism $f$, the stable and unstable bundles $E^{s}$ and $E^{u}$ are each $\beta$-H\"older continuous for some $\beta > 0$. In analogy to Definition \ref{defn: fiber bunched} we say that the derivative cocycle $Df|E^{u}$ is \emph{fiber-bunched} if there is an $N > 0$ such that 
\[
\|Df^{N}_{x}|E^{u}_{x}\| \cdot \|(Df^{N}_{x}|E^{u}_{x})^{-1}\| \cdot \text{max}\{\|Df^{N}_{x}|E^{s}_{x}\|^{\beta}, \|(Df^{N}_{x}|E^{u}_{x})^{-1}\|^{\beta}\} < 1.
\]
As in the case of a subshift of finite type, for each H\"older continuous potential $\varphi: M \rightarrow \mathbb{R}$ we have an equilibrium state $\mu_{\varphi}$ which is a fully supported ergodic invariant probability measure for $f$. The two most important equilibrium states for $f$ are the measure of maximal entropy (given by the potential $\varphi \equiv 0$) and the SRB measure characterized by having absolutely continuous conditional measures on the unstable leaves of $f$ (given by $\varphi(x) = -\log(|\det(Df_{x}|E^{u}_{x})|)$) which coincides with volume if $f$ is volume-preserving. To emphasize the dependence of $E^{u}$ on $f$ we will write $E^{u,f}$ for the unstable bundle associated to $f$.

\begin{corollary}\label{cor: deriv continuity}
Let $f: M \rightarrow M$ be a transitive $C^{1+\alpha}$ Anosov diffeomorphism for some $\alpha > 0$ and $\varphi: M \rightarrow \mathbb{R}$ a H\"older continuous potential. If $\dim E^{u} = 2$ and $Df|E^{u}$ is fiber-bunched then $f$ is a continuity point for the Lyapunov exponents $\lambda_{\pm}(Df|E^{u,f},\mu_{\varphi})$ as a function of $f \in \mathrm{Diff}^{1+\alpha}(M)$ and $\varphi \in C^{\beta}(M,\mathbb{R})$. 
\end{corollary}

\begin{proof}
Let $f_{k}$ be a sequence of $C^{1+\alpha}$-diffeomorphisms converging in $\mathrm{Diff}^{1+\alpha}(M)$ to $f$. For large enough $k$, $f_{k}$ is also an Anosov diffeomorphism, and moreover by structural stability there is a unique H\"older continuous homeomorphism $g_{k}: M \rightarrow M$ close to the identity such that $g_{k} \circ f_{k} = f \circ g_{k}$. Let $G_{k}:E^{u,f_{k}} \rightarrow E^{u,f}$ be a homeomorphism covering the homeomorphism $g_{k}: M \rightarrow M$ which is linear on the fibers $G_{k}(x): E^{u,f_{k}}_{x} \rightarrow E^{u,f}_{g_{k}(x)}$ and such that $G_{k}$ converges uniformly to the identity map on $E^{u,f}$ as $k \rightarrow \infty$. For $k$ sufficiently large we may take $G_{k}(x)$ to be the orthogonal projection of the plane $E^{u,f_{k}}_{x}$ onto $E^{u,f}_{g_{k}(x)}$. 

Since $Df|E^{u,f}$ is fiber-bunched and fiber bunching of $Df|E^{u}$ is an open condition in $\mathrm{Diff}^{1+\alpha}(M)$ we conclude that the cocycles $Df_{k}|E^{u,f_{k}}$ all admit stable and unstable holonomies $H^{s,k}$ and $H^{u,k}$ along the stable and unstable manifolds of $f_{k}$ and moreover that these stable and unstable holonomies converge locally uniformly to the stable and unstable holonomies $H^{s}$ and $H^{u}$ of $Df|E^{u}$ as the local stable and unstable manifolds of $f_{k}$ converge uniformly to those of $f$ (see \cite{BGV} and \cite{VianaAlmostAllCocyc}). We then define for each $k$ a new cocycle $A_{k}$ on the vector bundle $E^{u,f}$ by 
\[
A_{k}(x) = G_{k}(f_{k}(g_{k}^{-1}(x))) \circ Df_{k}(g_{k}^{-1}(x))|E^{u,f_{k}} \circ G_{k}^{-1}(g_{k}^{-1}(x))
\]
which admits stable and unstable holonomies 
\[
\widetilde{H}^{*,k}_{xy} = G_{k}(g_{k}^{-1}(y)) \circ H^{*,k}_{g_{k}^{-1}(x)g_{k}^{-1}(y)} \circ G_{k}^{-1}(g_{k}^{-1}(x))
\]
for $y \in W^{*}_{f}(x)$, $* =s,u$ and $W^{*}_{f}$ being the stable and unstable manifolds of $f$. Since $G_{k}$ converges uniformly to the identity on $E^{u,f}$ we conclude that $A_{k}$ converges to $Df|E^{u,f}$ uniformly and further that the stable and unstable holonomies of $A_{k}$ converge uniformly to those of $Df|E^{u,f}$.

The diffeomorphism $f$ admits a Markov partition and thus there is a subshift of finite type $\hat{f}: \hat{\Sigma} \rightarrow \hat{\Sigma}$ and a topological semiconjugacy $h: \hat{\Sigma} \rightarrow M$ such that $h \circ \hat{f} = f \circ h$ \cite{Bow2}. By refining the Markov partition if necessary, we can assume that the image of each cylinder $[0;j]$ of $\hat{\Sigma}$ under $h$ in $M$ is contained inside of an open set on which the bundle $E^{u,f}$ is trivializable. For $\hat{x} \in [0;j]$ let $L_{j}(\hat{x}):  \mathbb{R}^{2}  \rightarrow E^{u,f}_{h(\hat{x})}$ be the linear map associated to a fixed trivialization of $E^{u,f}$ over $h([0;j])$. We can then extend $h$ to a continuous surjection $L: \hat{\Sigma} \times \mathbb{R}^{2} \rightarrow E^{u,f}$,
\[
L(\hat{x},v) = [L_{j}(\hat{x})](v)
\]
which is a linear isomorphism on each of the fibers. We then define new linear cocycles $\hat{A}_{k}: \hat{\Sigma} \rightarrow GL(2,\mathbb{R})$ by 
\[
\hat{A}_{k}(\hat{x}) = L^{-1}(\hat{f}(\hat{x})) \circ A_{k}(h(\hat{x})) \circ L(\hat{x})
\]
which admit stable and unstable holonomies 
\[
\hat{H}^{*,k}_{\hat{x}\hat{y}} = L^{-1}(\hat{y}) \circ \widetilde{H}^{*,k}_{h(\hat{x})h(\hat{y})} \circ L(\hat{x})
\]
for $y \in W^{*}_{loc}(x)$, $* =s,u$. It is again clear that $\hat{A}_{k}$ converges to $\hat{A}$ uniformly and that the new stable and unstable holonomies $\hat{H}^{*,k}$ for $\hat{A}_{k}$ converge uniformly to those for $\hat{A}$.

Let $\nu_{k} = (g_{k})_{*}\mu_{\varphi_{k}}$. $\nu_{k}$ is the equilibrium state for $f$ associated to the potential $\varphi_{k} \circ g_{k}$ and thus is a fully supported ergodic $f$-invariant measure with local product structure on $M$. Let 
\[
\Omega = \{x \in M: \; \# h^{-1}(x) > 1\}.
\]
$\Omega$ is a null set for any equilibrium state associated to a H\"older continuous potential \cite{Bow2}. Hence we can lift $\nu_{k}$ to an $\hat{f}$-invariant measure $\hat{\nu}_{k}$ on $\hat{\Sigma}$ such that $h_{*}\hat{\nu}_{k} = \nu_{k}$. Furthermore $\hat{\nu}_{k}$ is the equilibrium state associated to the potential $\psi_{k} = \varphi_{k} \circ g_{k} \circ h$ on $\hat{\Sigma}$. As $k \rightarrow \infty$, $\psi_{k}$ converges in a H\"older norm to $\psi = \varphi \circ h$. It follows from Lemma \ref{lemma: equilibrium states} that $\hat{\nu}_{k}$ converges to $\hat{\nu}$ as in Section \ref{sec: prod}. 

Hence by the criterion of the Theorem \ref{mainthm} we get that $\lambda_{+}(\hat{A}_{k},\hat{\nu}_{k}) \to \lambda_{+}(\hat{A},\hat{\nu})$ and the same statement for $\lambda_{-}$. By construction the map $h: (\hat{\Sigma},\hat{\nu}_{k}) \rightarrow (M,\nu_{k})$ is a measurable isomorphism, and the same holds with $\hat{\nu}_{k}$ and $\nu_{k}$ replaced by $\hat{\nu}$ and $\nu$. Then by construction the map $L: \hat{\Sigma} \times \mathbb{R}^{2} \rightarrow E^{u,f}$ gives a measurable conjugacy between the cocycles $\hat{A}_{k}$ and $A_{k}$. It follows that $\lambda_{+}(\hat{A}_{k},\hat{\nu}_{k}) = \lambda_{+}(A_{k},\nu_{k})$ and $\lambda_{+}(\hat{A},\hat{\nu}) = \lambda_{+}(A,\nu) = \lambda_{+}(Df|E^{u,f},\mu_{\varphi})$. The map $g_{k}: (M,\varphi_{k}) \rightarrow (M,\nu_{k})$ is also a measurable isomorphism by construction and $G_{k}$ gives a measurable conjugacy from $Df_{k}|E^{u,f_{k}}$ to $A_{k}$ over this isomorphism. Hence we conclude that $\lambda_{+}(A_{k},\nu_{k}) = \lambda_{+}(Df_{k}|E^{u,f_{k}},\mu_{\varphi_{k}})$ for each $k$, which completes the proof. 
\end{proof}

By replacing $f$ with $f^{-1}$ we obtain the same corollary for $Df|E^{s}$ instead, provided that $\dim E^{s} = 2$. 

\begin{remark}
The conclusions of Corollary \ref{cor: deriv continuity} can be extended to 2-dimensional cocycles over maps $f: X \rightarrow X$ which are hyperbolic homeomorphisms (see \cite{AvilaVianaExtLyapInvPrin}) with $X$ a compact metric space. This includes the derivative cocycle of a diffeomorphism $f: M \rightarrow M$ over a hyperbolic set $\Lambda$ for $f$. Corollary \ref{cor: deriv continuity} can also be extended to the case of Anosov flows with 2-dimensional unstable bundle by using the fact that an Anosov flow is topologically semiconjugate via a Markov partition to a suspension flow over a subshift of finite type and then inducing on a transverse section to reduce to the case of a subshift of finite type.  
\end{remark}

\section{Preliminary Results}\label{sec: preliminary}

The rest of the paper is devoted to the proof of Theorem \ref{mainthm}. From now on $\hat{\mu}$ will denote an ergodic $\hat{f}$-invariant measure with local product structure and full support on $\hat{\Sigma}$. 
In this section we prove some preliminary results. 

\subsection{Projective cocycles}\label{sec:$su$-states} Let $\mathbb{P}^{1}$ be the 1-dimensional real projective space of lines in $\mathbb{R}^{2}$. Given a one-dimensional subspace $U\subset \R^2$ we will not distinguish $U\subset \R^2$ and $U\in \P^1$.  Given a non-zero vector $v\in \R^2$  we abuse notation and consider $v\in \Pone$ by identifying $v$ with its linear span.  
Given $T\in \Gl(2, \R)$ we write $\P T\colon \Pone\to \Pone$ for the induced projective map.  

Consider a cocycle $\hat A\colon \hat \Sigma \to \Gl(2, \R)$.  
The \textit{projective cocycle} associated to $\hat{A}$ and $\hat{f}$ is the map $\hat{F}_{\hat{A}}:\hat{\Sigma}\times \mathbb{P}^{1}\to  \hat{\Sigma}\times \mathbb{P}^{1}$ given by 
\begin{displaymath}
\hat{F}_{\hat{A}} (\hat{x},v)=(\hat{f}(\hat{x}),\P \hat{A}(\hat{x})v).
\end{displaymath}

\subsection{$s$- and $u$-states}
Let $\hat{m}$ be a probability measure on $\hat{\Sigma}\times  \mathbb{P}^{1}$  {projecting}  to $\hat{\mu}$;  that is, $\hat{\pi} _{\ast}\hat{m}=\hat{\mu}$ where $\hat{\pi} :\hat{\Sigma}\times  \mathbb{P}^{1}\rightarrow \hat{\Sigma}$ is the canonical projection. A \textit{disintegration} of $\hat{m}$ along the fibers is a measurable family $\lbrace \hat{m}_{\hat{x}}: \hat{x}\in \hat{\Sigma} \rbrace$ of probabilities on $\mathbb{P}^{1}$ satisfying 
\begin{displaymath}
\hat{m}(D)=\int _{\hat{\Sigma}} \hat{m}_{\hat{x}}(\lbrace v: (\hat{x}, v)\in D\rbrace) \ d\hat{\mu}(\hat{x})
\end{displaymath}
for any measurable set $D\subset \hat{\Sigma}\times \mathbb{P}^{1}$. Observe that $\hat{m}$ is $\hat{F}_{\hat{A}}$-invariant if and only if $\P \hat{A}(\hat{x})_{\ast}\hat{m}_{\hat{x}}=\hat{m}_{\hat{f}(\hat{x})}$ for $\hat{\mu}$-almost every $\hat{x}\in \hat{\Sigma}$. 

Following   \cite{AvilaVianaExtLyapInvPrin}  we say that a disintegration $\lbrace \hat{m}_{\hat{x}}: \hat{x}\in \hat{\Sigma} \rbrace$ of an $\hat{F}_{\hat{A}}$-invariant probability measure $\hat{m}$ projecting  to $\hat{\mu}$ is \textit{essentially s-invariant} with respect to a stable holonomy $H^{s,\hat{A}}$ for $\hat{A}$ if there is a full measure subset $E$ of $\hat{\Sigma}$ such that $\hat{x}, \hat{y} \in E$ and $\hat{y} \in W^{s}_{loc}(\hat{x})$ implies that
\begin{displaymath}
(H^{s,\hat{A}}_{\hat{x} \hat{y}}) _{\ast} \hat{m}_{\hat{x}}=\hat{m}_{\hat{y}} 
\end{displaymath}
We define the notion of an \textit{essentially u-invariant} disintegration similarly.  An $\hat{F}_{\hat{A}}$-invariant probability measure $\hat{m}$ projecting  to $\hat{\mu}$ is called an \textit{$s$-state} with respect to a stable holonomy $H^{s,\hat{A}}$ if it admits some disintegration which is essentially $s$-invariant.  We will always assume that the subset $E$ is \emph{s-saturated}, meaning that if $\hat{x} \in E$ then $W^{s}_{loc}(\hat{x}) \subset E$. This can always be done by modifying the disintegration of $\hat{m}$ on a $\hat{\mu}$-null set. We define \textit{$u$-states} similarly.

An $\hat{F}_{\hat{A}}$-invariant probability measure $\hat{m}$ is an  \textit{$su$-state} if it  is simultaneously an $s$-state and a $u$-state. The main property of $su$-states is the following.  

\begin{proposition}\label{Prop: continuous $su$-states}
Assume that $\hat{\mu}$ is fully supported and has local product structure. If $\hat{m}$ is an $su$-state then it admits a disintegration for which the conditional probabilities $\hat{m}_{\hat{x}}$ depend continuously on $\hat{x}$ and are both $s$-invariant and $u$-invariant. 
\end{proposition}
For a proof of this proposition see \cite[Proposition 4.8]{AvilaVianaExtLyapInvPrin}. 

 Given a cocycle with holonomies, there is always at least one $s$- and one $u$-state. On the other hand, $su$-states  impose some  rigidity on the system  as exhibited by Proposition \ref{Prop: continuous $su$-states} and as such need not always exist. 
However, here is  one situation in which $su$-states are guaranteed to exist.  
As stated here,  this follows from the main result in \cite{MR850070} and has been extended to  more general settings in \cite{AvilaVianaExtLyapInvPrin}.
\begin{theorem}[Invariance Principle] \label{theorem: invariance principle}
Let $\hat{A}:\hat{\Sigma} \rightarrow \Sl (2,\mathbb{R})$ be a  cocycle admitting stable and unstable holonomies  and assume that $\hat{\mu}$ is an ergodic $\hat{f}$-invariant probability measure with local product structure.  If $\lambda _+({\hat{A}}, \hat{\mu})=\lambda _-({\hat{A}}, \hat{\mu})= 0 $ then any $\hat{F}_{\hat{A}}$-invariant probability measure projecting to $\hat{\mu}$ is an $su$-state.
\end{theorem}

In the sequel, we will be interested in sequences of $s$- and $u$-states projecting to different base measures and invariant under different projective cocycles and corresponding holonomies.  
The next lemma gives a criterion for an accumulation point of such a sequence to be an $s$- or $u$-state for the limiting cocycle.  

\begin{lemma}\label{states} Let $\hat{A}_{k}\colon \hat{\Sigma} \rightarrow GL(2,\mathbb{R})$ be a sequence of linear cocycles with holonomies and suppose that $\hat{A}_{k}$ converges to $\hat{A}$ uniformly with holonomies. For each $k$ let $ \hat{m}_{k} $ be an  $s$-state for $ \hat{A}_{k} $ with respect to the stable holonomies $ H^{s,\hat{A}_{k}}   $ of $\hat{A}_{k}$ and projecting to a   fully supported $\hat{f}$-invariant probability measure  $ \hat{\mu}_{k} $ with local product structure. Suppose that the sequence $\hat{\mu}_{k}$ converges to $\hat{\mu}$ as in Section \ref{sec: prod} and that $\hat{m}_{k} \rightarrow \hat{m}$ in the weak-$*$ topology. Then $\hat{m}$ is an $s$-state with respect to the stable holonomies $H^{s,\hat{A}}$ for $\hat{A}$ which projects to $\hat{\mu}$. The same holds with unstable holonomies and u-states replacing stable holonomies and s-states. 
\end{lemma}

\begin{proof}
We will prove the statement for $s$-states. The statement for $u$-states then follows by considering the inverse cocycle $\hat{A}^{-1}$ over $\hat{f}^{-1}$. 

We begin by defining continuous changes of coordinates which make each $\hat{A}_{k}$ and $\hat{A}$ constant on local stable manifolds. For each $k$ let $\{\hat{m}^k_{\hat{x}}\}_{\hat{x}\in\hat{\Sigma}}$ be a disintegration of $\hat{m}_{k}$ along the $\mathbb{P}^{1}$ fibers. Each of these conditional measures is defined on a $\hat{\mu}_{k}$-full measure set $E_{k} \subset \hat{\Sigma}$ which we may assume to be $s$-saturated, since these measures are $s$-states, and we may assume these conditional measures are invariant under stable holonomy on $E_{k}$. We may also assume that the sets $E_{k}$ are $\hat{f}$-invariant. 

Fix $\ell$ points $\hat{z}_{1},\dots,\hat{z}_{\ell}$ with $\hat{z}_{i} \in [0;i]$. For $\hat{x} \in [0;i]$, let $g(\hat{x})$ be the unique point in the intersection $W^{u}_{loc}(\hat{z}_{i}) \cap W^{s}_{loc}(\hat{x})$. Note that $g(\hat{x}) = g(\hat{y})$ if $\hat{y} \in W^{s}_{loc}(\hat{x})$. Define
\[
\widetilde{A}_{k}(\hat{x}) = H^{s,\hat{A}_{k}}_{g(\hat{f}(\hat{x}))\hat{f}(\hat{x})} \circ \hat{A}_{k}(g(\hat{x})) \circ H^{s,\hat{A}_{k}}_{ \hat{x} g(\hat{x})}
\]
for each $k$, and define $\widetilde{A}$ similarly. By construction each $\widetilde{A}_{k}$ is constant along local stable manifolds and furthermore, since the stable holonomies $H^{s,\hat{A}_{k}}$ converge uniformly to $H^{s,\hat{A}}$, we also have that $\widetilde{A}_{k} \rightarrow \widetilde{A}$ uniformly.

Define $\hat{\nu}^k_{\hat{x}} = (H^{s,\hat{A}_{k}}_{ \hat{x} g(\hat{x})})_{*}\hat{m}^k_{\hat{x}}$ and let $\hat{\nu}_{k}$ be the probability measure on $\hat{\Sigma} \times \mathbb{P}^{1}$ projecting to $\hat \mu_k$ with this disintegration along the $\mathbb{P}^{1}$ fibers. $\hat{\nu}_{k}$ is $\hat{F}_{\widetilde{A}_{k}}$-invariant and since the linear maps $H^{s,\hat{A}_{k}}_{ \hat{x} g(\hat{x}) }$ depend continuously on $\hat{x}$, we conclude that $\hat{\nu}_{k}$ converges in the weak-$*$ topology to a measure $\hat{\nu}$ with disintegration $\hat{\nu}_{\hat{x}} = (H^{s,\hat{A}}_{ \hat{x} g(\hat{x}) })_{*}\hat{m}_{\hat{x}}$. To prove that $\hat{m}$ is an $s$-state it thus suffices to show that for $\hat{\mu}$-a.e.\ pair of points $\hat{x}$ and $\hat{y}$ with $\hat{y} \in W^{s}_{loc}(\hat{x})$ we have $\hat{\nu}_{\hat{x}} = \hat{\nu}_{\hat{y}}$, because if $\hat{y} \in W^{s}_{loc}(\hat{x})$ for some $\hat{x}$ in the intersection of this full measure subset with $E$, we then have 
\[
(H^{s,\hat{A}}_{ \hat{x} g(\hat{x}) })_{*}\hat{m}_{\hat{x}} = \hat{\nu}_{\hat{x}} = \hat{\nu}_{\hat{y}} = (H^{s,\hat{A}}_{ \hat{y} g(\hat{y}) })_{*}\hat{m}_{\hat{y}}
\]
and therefore
\[
\hat{m}_{\hat{y}} = \left(H^{s,\hat{A}}_{g(\hat{y})\hat{y}} \circ H^{s,\hat{A}}_{\hat{x}g(\hat{x})}\right)_{*}\hat{m}_{\hat{x}} = \left(H^{s,\hat{A}}_{g(\hat{x})\hat{y}} \circ H^{s,\hat{A}}_{\hat{x}g(\hat{x})}\right)_{*}\hat{m}_{\hat{x}} = \left(H^{s,\hat{A}}_{\hat{x}\hat{y}}\right)_{*}\hat{m}_{\hat{x}} 
\]
where we used $g(\hat{x}) = g(\hat{y})$ in the second line. Since the measures $\hat{\nu}_{k}$ are $s$-states we have $\hat{\nu}^k_{\hat{y}} = \hat{m}^k_{g(\hat{x})}$ for every $\hat{y} \in W^{s}_{loc}(\hat{x})$, so the disintegrations of the measures $\hat{\nu}_{k}$ are constant on $\hat{\mu}_{k}$-a.e.\ local stable manifold. 

There are continuous maps $A_{k}: \Sigma^{u} \rightarrow GL(2,\mathbb{R})$ such that $A_{k} \circ P^{u} = \widetilde{A}_{k}$ and such that $A_{k} \rightarrow A$ uniformly, where $A \circ P^{u} = \widetilde{A}$. Let $\nu_{k}$, $\nu$ be the images of the measures $\hat{\nu}_{k}$, $\hat{\nu}$ under the projection $P^{u} \times Id: \hat{\Sigma} \times \mathbb{P}^{1} \rightarrow \Sigma^{u} \times \mathbb{P}^{1}$. The disintegration $\{\hat{\nu}^k_{\hat{x}}\}_{\hat{x} \in \hat{\Sigma}}$ descends under this projection to a disintegration $\{\nu ^k_{x}\}_{x \in \Sigma^{u}}$ with the property that for $\mu_{k}^{u}$-a.e.\ $x$, 

\begin{displaymath}
A_{k}(x)_{*}\nu ^k_{x} = \nu ^k_{f_{u}(x)}.
\end{displaymath}

We first show that $A(x)_{*}\nu_{x} = \nu_{f_{u}(x)}$ for $\mu^{u}$-a.e.\ $x\in \Sigma ^u$. Let $\eta$ be the probability measure on $\Sigma^{u} \times \mathbb{P}^{1}$ with disintegration $\{A^{-1}(x)_{*}\nu_{f_{u}(x)}\}_{x \in \Sigma^{u}}$. It suffices for this claim to prove that $\eta = \nu$, since the disintegration of $\nu$ along the $\Pone$ fibers is unique up to $\mu^{u}$-null sets. Let $\varphi: \Sigma^{u} \times \mathbb{P}^{1} \rightarrow \mathbb{R}$ be a continuous function and define
\begin{displaymath}
\Phi(x) = \int_{\mathbb{P}^{1}}\varphi(x,A^{-1}(x)v)\,d\nu_{f_{u}(x)}(v).
\end{displaymath} 
Since $\mu^{u}$ is $f_{u}$-invariant and admits a positive Jacobian $J_{\mu ^u}f_{u}$ with respect to $f_u$, 
\begin{align*}
\int_{\Sigma^{u}}\Phi(x) \,d\mu^{u}(x) &= \int_{\Sigma^{u}}\left(\sum_{y \in f_{u}^{-1}(x)}\frac{1}{J_{\mu ^u}f_{u}(y)}\Phi(y)\right) \,d\mu^{u}(x) \\
&=\int_{\Sigma^{u}}\int_{\mathbb{P}^{1}}\sum_{y \in f_{u}^{-1}(x)}\frac{1}{J_{\mu ^u}f_{u}(y)}\varphi(y,A^{-1}(y)v)\,d\nu_{x}(v) \,d\mu^{u}(x) \\
&= \int_{\Sigma^{u}\times\mathbb{P}^{1}}\sum_{y \in f_{u}^{-1}(x)}\frac{1}{J_{\mu ^u}f_{u}(y)}\varphi(y,A^{-1}(y)v)\,d\nu(x,v).
\end{align*} 
On the other hand, 
\begin{align*}
\int_{\Sigma^{u}}\Phi(x) \,d\mu^{u}(x) &= \int_{\Sigma^{u}}\int_{\mathbb{P}^{1}}\varphi(x,A^{-1}(x)v)\,d\nu_{f_{u}(x)}(v)\,d\mu^{u}(x) \\
&= \int_{\Sigma^{u}\times\mathbb{P}^{1}}\varphi(x,v)\, d\eta(x,v).
\end{align*}
Hence it suffices to show that for every continuous map $\varphi: \Sigma^{u} \times \mathbb{P}^{1} \rightarrow \mathbb{R}$ we have
\begin{displaymath}
\int \varphi \,d\nu = \int \sum_{y \in f_{u}^{-1}(x)}\frac{1}{J_{\mu ^u}f_{u}(y)}\varphi(y,A^{-1}(y)v)\,d\nu(x,v).
\end{displaymath}
But for each $k$ we know that for $\mu^{u}_{k}$-a.e.\ $x \in \Sigma^{u}$ we have $A_{k}^{-1}(x)_{*}\nu ^k_{f_{u}(x)} = \nu ^k_{x}$. The same calculation as above shows that the above equality holds with appropriate modifications for $\nu_{k}$, i.e., 
\begin{displaymath}
\int \varphi \,d\nu_{k} = \int \sum_{y \in f_{u}^{-1}(x)}\frac{1}{J_{\mu ^u_{k}}f_{u}(y)}\varphi(y,A_{k}^{-1}(y)v)\,d\nu_{k}(x,v)
\end{displaymath}
By assumption, $\nu_{k}$ converges to $\nu$ in the weak-$*$ topology, $A_{k}^{-1} \rightarrow A^{-1}$ uniformly, and $J_{\mu ^u_{k}}f_u \rightarrow J_{\mu ^u}f_u$ uniformly. It follows that this equality holds in the limit $k \rightarrow \infty$, and hence that $A(x)_{*}\nu_{x} = \nu_{f_{u}(x)}$ for $\mu^{u}$-a.e.\ $x$.

The disintegration of the measure $\hat{\nu}$ along the $\Pone$ fibers of $\hat{\Sigma} \times \Pone$ can be recovered from the disintegration of $\nu$ along the $\Pone$ fibers of $\Sigma^{u} \times \Pone$ by the formula
\[
\hat{\nu}_{\hat{x}} = \lim_{n \rightarrow \infty} A^{n}(P^{u}(\hat{f}^{-n}(\hat{x})))_{*}\nu_{P^{u}(\hat{f}^{-n}(\hat{x}))}
\] 
(see Lemma 3.4 of \cite{AvilaVianaExtLyapInvPrin}) for $\hat{\mu}$-a.e.\ $\hat{x}$. But we have just shown that 
\[
A^{n}(P^{u}(\hat{f}^{-n}(\hat{x})))_{*}\nu_{P^{u}(\hat{f}^{-n}(\hat{x}))} = \nu_{P^{u}(\hat{x})}
\]
for every $n$. Hence we conclude that 
\[
\hat{\nu}_{\hat{x}} = \nu_{P^{u}(\hat{x})}
\]
and thus $\hat{\nu}_{\hat{x}} = \hat{\nu}_{\hat{y}}$ for $\hat{y} \in W^{s}_{loc}(\hat{x})$. 
\end{proof}

\subsection{Continuity of conditional measures} \label{sec: continuity conditional measures}

From now on we will write $\Sigma$, $f$, $P$ and $\mu$ for $\Sigma^{u}$, $f_u$, $P^{u}$, and $\mu^{u}$, respectively. Moreover, from the proof of Lemma \ref{states} it follows that an arbitrary sequence of cocycles $\{\hat{A}_{k}\}_{k \in \mathbb{N}}$ converging uniformly with holonomies to a cocycle $\hat{A}$ may be straightened out using the stable holonomies so that each $\hat{A}_{k}$ and $\hat{A}$ are constant on local stable sets and the property of uniform convergence is preserved.  Moreover, the straightened out cocycles still admit $u$-holonomies and the $u$-holonomies also converge uniformly.  

Consider such a cocycle $\hat A$ that has been straightened out along stable holonomies.  
We write  $A:\Sigma \to \Gl(2, \R)$ for the continuous map defined by $\hat A = A\circ P$. In particular, $A(x)=\hat{A}(\hat{x})$ for every $\hat{x}\in W^s_{loc}(x)$.

\subsubsection{Measures induced from a $u$-state}
In the sequel, we will be primarily interested in families of measure on $\Sigma\times \Pone$ induced from measures on $\hat \Sigma\times \Pone$ with  certain dynamical properties.  The measures on $\Sigma\times \Pone$ will in turn have certain geometric properties that we describe here.

\begin{definition}\label{def:meas}
A probability measure $m$ on $\Sigma \times \Pone$  is said to be \emph{induced from a $u$-state} if there exists 
\begin{itemize}
\item a cocycle $\hat A\colon \hat \Sigma \to \Gl(2, \R)$ that is constant along local stable manifolds and  admits a continuous family of unstable holonomies $H^{u, \hat A}$, 
\item a fully supported measure $\hat \mu$ on $\hat \Sigma$ with local product structure, 
\item and an $\hat{F}_{\hat{A}}$-invariant measure $\hat{m}$ on $\hat{\Sigma} \times \Pone$  projecting to  $\hat{\mu} $ such that $\hat{m}$ is a $u$-state for the holonomies $H^{u,\hat A}$ with $m=(P\times Id)_{\ast}\hat{m}.$
\end{itemize}
\end{definition}
Note that such an $m$ is necessarily $F_A$-invariant, where $A$ is such that $\hat A = A\circ P$ as above. 

\subsubsection{Continuity of the disintegration of  measures induced from $u$-states}
They key geometric fact we exploit in the remainder of the paper is  that every measure $m$ induced from a $u$-state admits a disintegration into a \emph{continuous} family of conditional measures $\lbrace m_{x}: x\in \Sigma \rbrace$.  The continuity properties of the conditional measures of $m$ were first established in \cite{BV}; in this section we establish additional equicontinuity properties of the conditional measures over families of linear cocycles on which  unstable holonomies exist and vary continuously. 

We retain all notation from Definition  \ref{def:meas}.
Observe that if 
$m=(P\times Id)_{\ast}\hat{m}$ and  $\lbrace \hat{m}_{\hat{x}}: \hat{x}\in \hat{\Sigma} \rbrace$ is a disintegration of $\hat{m}$ along the fibers $\lbrace \hat{\pi}^{-1}(\hat{x}); \hat{x}\in \hat{\Sigma}\rbrace$ 
then for $x\in \Sigma$ 
\begin{equation}\label{eq:thisgoose}
m_x=\int _{W^s_{\loc}(x)} \hat{m}_{\hat{x}} \ d\hat{\mu}_x(\hat{x})
\end{equation}
is a disintegration of $m$ relative to $\lbrace \pi ^{-1}(x); x\in \Sigma\rbrace$ where $\pi :\Sigma\times  \mathbb{P}^{1}\rightarrow \Sigma$ is the canonical projection.

\begin{proposition}\label{prop: continuity of A mu invariant}
Any probability measure $m$ induced from a $u$-state admits a disintegration into conditional measures $\{m_x\}_{x\in \Sigma}$ that are defined for every $x\in \Sigma$ and vary continuously with $x$ in the weak-$*$ topology.  
\end{proposition}

\begin{proof}
Let $\hat{m}$ be a $u$-state such that $(P\times Id)_{\ast}\hat{m}= m$ and $\{\hat{\mu}_x\}_{x\in \Sigma}$ a disintegration of $\hat{\mu}$ as in Lemma \ref{lemma:continuity mu}. Take a disintegration $(\hat{m}_{\hat{x}})_{\hat{x}\in \hat{\Sigma}}$ of $\hat{m}$  such that for $\hat{\mu}$-a.e.\ $\hat{x} \in \hat{\Sigma}$,
$$(H^{u,\hat{A}}_{\hat{x} \hat{y}}) _{\ast} \hat{m}_{\hat{x}}=\hat{m}_{\hat{y}} \; \textrm{for every} \; \hat{y}\in W^u_{\loc}(\hat{x})
$$
and let $\{m_x\}_{x\in \Sigma}$ be the disintegration of $m$ as in \eqref{eq:thisgoose}. 

Let $g\colon \Pone \rightarrow \mathbb{R} $ be  continuous  and consider $x,y\in \Sigma$ in the same cylinder $[0;i]$. Then, changing variables $\hat{y}=h_{x,y}(\hat{x})$ we get that
\begin{displaymath}
\begin{split}
\int _{\Pone} g dm_y &=\int _{W^s_{\loc}(y)} \int _{\Pone} g d\hat{m}_{\hat{y}} d\hat{\mu}_y(\hat{y})\\
&= \int _{W^s_{\loc}(x)} \biggl( \int _{\Pone} g\circ H^{u,\hat{A}}_{\hat{x}\hat{y}}d\hat{m}_{\hat{x}} \biggr) { R_{x,y}}(\hat{x})d\hat{\mu}_x(\hat{x})
\end{split}
\end{displaymath}
since $\hat{m}$ is an u-sate. Thus,
\begin{displaymath}
\left| \int g dm_y - \int g dm_x \right| \leq \int _{W^s_{\loc}(x)}  \int _{\Pone} \left| g\circ H^{u,\hat{A}}_{\hat{x}\hat{y}}\cdot R_{x,y}(\hat{x}) -g\right| d\hat{m}_{\hat{x}}d\hat{\mu}_x(\hat{x}).
\end{displaymath}

From the continuity properties of unstable holonomies (see Definition \ref{defn: invariant holonomies}) we have that $\| H^{u,\hat{A}}_{\hat{x}\hat{y}} -Id \|$  is uniformly close to zero whenever $x$ and $y$ are close. Moreover, Lemma \ref{lemma:continuity mu} implies that $\| R_{x,y}-1\|_{L^{\infty}}$ is also close to zero whenever $x$ and $y$ are close. Therefore, given $\varepsilon >0$ there exist $\gamma >0$ such that $d(x,y)<\gamma$ implies $\| g\circ H^{u,\hat{A}}_{\hat{x}\hat{y}}\cdot R_{x,y}(\hat{x}) -g\| _{L^{\infty}} <\varepsilon$ and thus $\mid \int g dm_y - \int g dm_x \mid <\varepsilon $ as we want.

\end{proof}

\begin{remark}\label{remark: invariant measure}
A probability measure $m$ in $\Sigma \times \Pone$ is $F_A$-invariant if and only if 
\begin{equation}\label{eq:lammababies}
\sum _{y\in f^{-1}(x)} \dfrac{1}{J_{\mu}f(y)}A(y)_{\ast}m_y=m_x
\end{equation}
for $\mu$ almost every $x\in \Sigma$ and any disintegration $\{m_x\}_{x\in \Sigma}$. When $m$ is induced from a $u$-state and $\{m_x\}_{x\in \Sigma}$ is the  the continuous family of  conditional measures above then  \eqref{eq:lammababies} holds for \textit{every} $x\in \Sigma$.
\end{remark}

We recall the setting of Lemma \ref{states}. 
Let $\hat \mu_k$ be a family of fully supported measures on $\hat \Sigma$ with product structure.  Assume $\hat \mu_k$ converges as in Section \ref{sec: prod} to a  fully supported measure $\hat \mu$ with product structure. In particular, the family of Jacobians $R^k_{x,y}$ associated to the disintegration of $\hat{\mu}_k$ given by Lemma \ref{lemma:continuity mu} converge uniformly to the Jacobians $R_{x,y}$ of $\hat{\mu}$.

For each $k$ let $\hat{A}_k$ be a cocycle that is constant along stable manifolds, and suppose $\hat{A}_k\to \hat{A}$ uniformly.  Moreover assume $\hat{A}_k$ and $\hat{A}$ admits  (unstable) holonomies and that $H^{u,\hat{A}_k}$ converges to $H^{u,\hat{A}}$ as in Section \ref{sec: holonomies}.  
For each $k$, let $m_k$ be a measure on $\Sigma\times \Pone$ induced by a $u$-state $\hat m_k$ for the holonomies $H^{u, \hat{A}_k}$ and projecting to  $\mu_k$.  Assume that $\hat m_k$ converges in the weak-$*$ topology to $\hat m$. From Lemma \ref{states} we have that $\hat m$ is a $u$-state for the holonomies $H^{u,\hat{A}}$ and projects to $\hat{\mu}$.  
Let  $m = (P\times Id)_{\ast}\hat{m}$ be the measure induced by the $u$-state $\hat m$.  

Observing that all the convergences above are uniform and following the same lines as in the proof of the previous proposition we get

\begin{proposition}\label{prop: continuity of Ak mu invariant}
The measures $m_k$ and $m$  admit disintegrations into conditional measures $\{m^k_x\}_{x\in \Sigma}$ and $\{m_x\}_{x\in \Sigma} $, respectively, which are defined for every $x\in \Sigma$ and such that the family $\{\{m^k_x\}_{x\in \Sigma},\{m_x\}_{x\in \Sigma} \}_k$ is equicontinuous. More precisely, for every continuous function $g:\Pone\rightarrow \mathbb{R} $ and $\varepsilon >0$ there exists $\delta >0$ such that $d_{\theta}(x,y)<\delta$ implies $\mid \int gdm_x -\int gdm_y\mid <\varepsilon$ and $\mid \int gdm^k_x -\int gdm^k_y\mid <\varepsilon$ for every $k\in \mathbb{N}$.

\end{proposition}

    Let $\{m^k_x\}_{x\in \Sigma}$ and $\{m_x\}_{x\in \Sigma}$ be the  continuous family of  conditional measures  constructed above.

\begin{lemma}\label{lemma:convergence of mkx to mx}
For any $x\in \Sigma$, $m^k_x\to m_x$.  Moreover, the convergence is uniform in $x$.
\end{lemma}
\begin{proof}
Let $g\colon \Pone\rightarrow \mathbb{R}$ be  continuous  and $\varepsilon >0$. By Proposition \ref{prop: continuity of Ak mu invariant}, there exists $\delta >0 $ such that, if $d_{\theta}(x,y)\leq \delta$ then
\begin{equation}\label{eq: convergence of mkx to mx}
\begin{split}
&\left| \int _{\Pone} g dm_x -\int _{\Pone} g dm_y \right| < \dfrac{\varepsilon}{10} \\
\textrm{and} \quad  \\
&\left| \int _{\Pone} g dm^k_x -\int _{\Pone} g dm^k_y \right| < \dfrac{\varepsilon}{10} \\
\end{split}
\end{equation}
for every $k\in \mathbb{N}$. Cover $\Sigma$ with finitely many clopen sets $V_i$ such that $\textrm{diam}(V_i)<\delta$. As $m_k$ converges to $m$ there exists $k_0\in \mathbb{N}$ such that 
\begin{equation}\label{eq: convergence of mkx to mx 2}
\left| \int _{V_i} \left( \int _{\Pone} g dm^k_x \right) d\mu_k (x)- \int _{V_i} \left( \int _{\Pone} g dm_x \right) d\mu (x) \right| < \frac{\varepsilon \mu ({V_i})}{10}
\end{equation}
and taking $M =\max\{1, \max |g|\}$ 
$$ \left | 1- \frac{\mu_k(V_i)}{\mu(V_i)}\right|\le \frac{\varepsilon}{10M}$$
for every $k\geq k_0$ and each $V_i$.  

Given $x\in \Sigma $  take $V_i$ with $x\in V_i$.  Then 
\begin{align*}
\left| \int _{\Pone} g   \ dm^k_x -\int _{\Pone} g   \ dm_x \right| &\\
=&\dfrac{1}{\mu (V_i)}  \left|  \int _{V_i} \left( \int _{\Pone} g   \ dm^k_x \right) d\mu (y)- \int _{V_i} \left( \int _{\Pone} g   \ dm_x \right) d\mu (y) \right| \\
\leq&
 \dfrac{1}{\mu (V_i)} \left| \int _{V_i}  \int _{\Pone} g   \ dm^k_x  d\mu (y)- \int _{V_i} \int _{\Pone} g   \ dm^k_x  d\mu_k (y)   \right| \\
&+ \dfrac{1}{\mu (V_i)}  \left|  \int _{V_i}  \int _{\Pone} g   \ dm^k_x  d\mu_k (y)- \int _{V_i}  \int _{\Pone} g   \ dm^k_y  d\mu_k (y)   \right| \\  
&+ \dfrac{1}{\mu (V_i)} \left|   \int _{V_i}  \int _{\Pone} g   \ dm^k_y  d\mu_k (y)- \int _{V_i} \int _{\Pone} g   \ dm_y d\mu (y)  \right| \\
& + \dfrac{1}{\mu (V_i)} \left| \int _{V_i}  \int _{\Pone} g   \ dm_y d\mu (y)- \int _{V_i}  \int _{\Pone} g   \ dm_x d\mu (y) \right| \\
\leq& 
 M\left |1 - \frac{\mu_k(V_i)}{\mu(V_i)}\right|\\
&+ \dfrac{1}{\mu (V_i)} \left( \int _{V_i}  \int _{\Pone} \left| g   \ dm^k_x -   g   \ dm^k_y \right| d\mu _k(y) \right) \\
 &+  \dfrac{1}{\mu (V_i)} \left| \int _{V_i}  \int _{\Pone} g   \ dm^k_y  d\mu _k(y)- \int _{V_i} \int _{\Pone} g   \ dm_y d\mu (y) \right|  \\
 &+\dfrac{1}{\mu (V_i)} \left( \int _{V_i}  \int _{\Pone}  \left| g \ dm_y -  g \  dm_x \right| d\mu (y) \right) 
\\ \leq &\dfrac{ \varepsilon}{10} +\left(1 +\dfrac{ \varepsilon}{M10} \right) \dfrac{ \varepsilon}{10} + \dfrac{ \varepsilon}{10}+ \dfrac{ \varepsilon}{10} 
\\ \leq &\dfrac{5 \varepsilon}{10}. \qedhere\end{align*}
\end{proof}

\section{Reductions in the proof of  Theorem \ref{mainthm}} \label{sec: proof of the main theorem}
We begin the proof of Theorem \ref{mainthm}. We start by observing that it suffices to prove continuity for cocycles taking values in $SL(2,\mathbb{R})$ instead of $GL(2,\mathbb{R})$. By continuity of $\hat{A}$ and compactness of $\hat{\Sigma}$, the function $s(\hat{x}) = \text{sgn}(\det(A(x)))$ is continuous on $\hat{\Sigma}$. Given $\hat{A}:\hat{\Sigma} \rightarrow GL(2,\mathbb{R})$ consider $g_{\hat{A}}:\hat{\Sigma}\rightarrow \mathbb{R}$ defined by $g_{\hat{A}}(\hat{x})=s(\hat{x})(|\det \hat{A}(\hat{x})|)^{\frac{1}{2}}$ and $\hat{B}:\hat{\Sigma} \rightarrow SL(2,\mathbb{R})$ such that $\hat{A}(\hat{x})=g_{\hat{A}}(\hat{x})\hat{B}(\hat{x})$. Thus, since

\begin{displaymath}
\lambda ^{\pm}(\hat{A},\hat{\mu})=\lambda ^{\pm}(\hat{B},\hat{\mu}) + \int \log |g_{\hat{A}}(\hat{x})|\  d\hat{\mu}(\hat{x}),
\end{displaymath}
and $g_{\hat{A}_{k}} \rightarrow g_{\hat{A}}$ uniformly, we get that $\lambda ^{\pm}(\hat{A}_{k},\hat{\mu}_{k}) \rightarrow \lambda ^{\pm}(\hat{A},\hat{\mu})$ if and only if $\lambda ^{\pm}(\hat{B}_{k},\hat{\mu}_{k}) \rightarrow \lambda ^{\pm}(\hat{B},\hat{\mu})$ where $\hat{B}_{k}$ is defined analogously to $\hat{B}$ for $\hat{A}_{k}$. Moreover, 
\begin{displaymath}
\lambda ^{+}(\hat{A},\hat{\mu})=\lambda ^{-}(\hat{A},\hat{\mu}) \Longleftrightarrow \lambda ^{+}(\hat{B},\hat{\mu})=0=\lambda ^{-}(\hat{B},\hat{\mu}).
\end{displaymath}
From now on, we will assume that our cocycles always take values in $SL(2,\mathbb{R})$.

The  proof of Theorem \ref{mainthm} is by contradiction.
Suppose  $(\hat A, \hat\mu, H^{s,\hat A}, H^{u,\hat A})$ and $(\hat A_k, \hat\mu_k, H^{s,\hat A_k}, H^{u,\hat A_k})$ are as in Theorem  \ref{mainthm}. Moreover, suppose for the purposes of contradiction that \begin{equation}\label{eq:mainassump}\lambda_+(\hat A_k,\hat\mu_k) \not \to \lambda_+(\hat A,\hat\mu).\end{equation}  We then also have $\lambda_-(\hat A_k,\hat\mu_k) \not \to \lambda_-(\hat A,\hat\mu).$

\subsection{Characterization of discontinuity points}\label{sec: discontinuity points} 
From \cite[Lemma 9.1]{V2}  we have that the functions $(\hat B, \hat \nu)\mapsto \lambda_+(\hat B, \hat \nu)$ and $(\hat B, \hat \nu)\mapsto \lambda_-(\hat B, \hat \nu)$ are, respectively, upper- and lower-semicontinuous with respect to the  topology of uniform convergence on continuous cocycles $\hat B$ and weak-$*$ convergence in $\hat \nu$.   Thus, assuming \eqref{eq:mainassump} we may assume 
$\lambda_-({\hat{A}},\hat{\mu})<0 <\lambda _+({\hat{A}},\hat{\mu})$.

Let $\mathbb{R}^2= E^{s,\hat{A}}_{\hat{x}}\oplus E^{u,\hat{A}}_{\hat{x}}$
be the Oseledets decomposition associated to $\hat{A}$ at the point $\hat x\in \hat \Sigma$. Consider the measures on $\hat \Sigma \times \Pone$ defined by 
\begin{displaymath}
\hat{m}^s=\int _{\hat{\Sigma}}\delta _{(\hat{x},E^{s,\hat{A}}_{\hat{x}})} d\hat{\mu}(\hat{x})
\quad \textrm{and} \quad
\hat{m}^u=\int _{\hat{\Sigma}}\delta _{(\hat{x},E^{u,\hat{A}}_{\hat{x}})} d\hat{\mu}(\hat{x}).
\end{displaymath}

By construction, $\hat{m}^s$ and $\hat{m}^u$ are $\hat{F}_{\hat{A}}$-invariant probability measures with projections $\hat{\mu}$. Moreover, $\hat{m}^s$ is an $s$-state (with disintegration $\{\delta_{E^{s,\hat A}_{\hat x}}\}_{\hat x\in \hat \Sigma}$) and $\hat{m}^u$ is a $u$-state.  By the Birkhoff ergodic theorem
\begin{displaymath}
\lambda _-(\hat{A},\hat{\mu}) =\int _{\hat{\Sigma}\times \Pone} \Phi _{\hat{A}}(\hat{x},v) \ d\hat{m}^s (\hat{x}, v)
\end{displaymath}
and
\begin{displaymath}
\lambda _+(\hat{A}, \hat{\mu}) =\int _{\hat{\Sigma}\times \Pone} \Phi _{\hat{A}}(\hat{x},v) \ d\hat{m}^u (\hat{x}, v)
\end{displaymath}
where $$ \Phi _{\hat{A}}(\hat{x},v) = \frac{\|\hat A(\hat x)(v)\|}{\|v\|}.$$

By the (non-uniform) hyperbolicity of $(\hat{A},\hat{\mu})$ we have the following.  
\begin{claim}\label{lemma:convex combination}
Let $\hat{m}$ be a probability measure on $\hat{\Sigma}\times \Pone$ projecting  to $\hat{\mu}$. Then, $\hat{m}$ is $\hat{F}_{\hat{A}}$-invariant if and only if it is a convex combination of $\hat{m}^s$ and $\hat{m}^u$: $\hat m = \alpha \hat m^s + \beta \hat m^u$ where $\alpha$ and $\beta$ are constant. 
\end{claim}
Indeed, one only has to note that every compact subset of $\Pone$ disjoint from $\lbrace E^u, E^s\rbrace$ accumulates on $E^u$ in the future and on $E^s$ in the past.  That $\alpha$ and $\beta$ are constant (independent of $\hat x\in \hat \Sigma$) follows from ergodicity.

We  now prove the key characterization of discontinuity points for the extremal  Lyapunov exponents.  The proof is well known but is included here for completeness.  
\begin{proposition}\label{prop: characterization of disc. points}
If $(\hat{A}, \hat \mu)$ is as in \eqref{eq:mainassump}, then every $\hat{F}_{\hat{A}}$-invariant probability measure $\hat{m}$ on $\hat \Sigma\times \Pone$ projecting to  $\hat{\mu}$ is an $su$-state for $\hat{F}_{\hat{A}}$.
\end{proposition}

\begin{proof} By the upper semi-continuouity of $\lambda_+(\cdot, \cdot)$, passing to a subsequence we may assume  $\lim _{k\rightarrow \infty} \lambda _+({\hat{A}_k}, \hat{\mu}_k)<\lambda _+({\hat{A}}, \hat{\mu})$. 
For each $k$, there 
exists an ergodic,  $\hat{F}_{\hat{A}_k}$-invariant probability measure $\hat{m}_k$, projecting to $\hat \mu_k$, which is a $u$-state for $H^{u,\hat A_k}$, and such that 
\begin{displaymath}
\lambda _+(\hat{A}_k, \hat{\mu}_k)=\int _{\hat \Sigma\times \Pone}\Phi _{\hat A_k} d\hat{m}_k.
\end{displaymath}
Indeed, if 
$\lambda_+({\hat{A}_k},\hat{\mu}_k)\neq 0$ we can take $\hat{m}_k=\int _{\hat{\Sigma}}\delta _{(\hat x,E^{u,\hat{A}_k}_{\hat{x}})} d\hat{\mu}_k(\hat{x})$ as  above. 
If $\lambda_+({\hat{A}_k},\hat{\mu}_k)=0$ then (as $\hat A_k\in SL(2,\R)$) we have $ \lambda_-({\hat{A}_k},\hat{\mu}_k)=0$ and by Theorem \ref{theorem: invariance principle}, any $\hat{F}_{\hat{A}_k}$-invariant probability measure $\hat{m}_k$, projecting to $\hat \mu_k$ is a $su$-state; moreover for any such measure $\int _{\hat \Sigma\times \Pone}\Phi _{\hat A_k} d\hat{m}_k = 0.$

Consequently 
\begin{displaymath}
\lim _{k\rightarrow \infty} \int _{\hat{\Sigma}\times \Pone} \Phi _{\hat A_k} \ d\hat{m}_k < \lambda _+ (\hat{A},\hat{\mu}).
\end{displaymath}
Taking subsequences again, we may assume that $(\hat{m}_k)_k$ converges to a  $\hat{F}_{\hat{A}}$-invariant probability measure $\hat{m}$. By Lemma \ref{states}, $\hat{m}$ is a  $u$-state for $H^{u, \hat A}$.  Now, by Claim \ref{lemma:convex combination},
\begin{displaymath}
\hat{m} =\alpha \hat{m}^s + \beta \hat{m}^u
\end{displaymath}
for some constants $\alpha,\beta \in [0,1]$.  
By uniform convergence of $\Phi_{\hat A_k}\to \Phi_{\hat A}$ and    weak-$*$  convergence of $\hat m_k\to \hat m$ we have
\begin{displaymath}
\int _{\hat{\Sigma}\times \Pone} \Phi _{\hat{A}} \ d\hat{m} = \lim _{k\rightarrow \infty} \int _{\hat{\Sigma}\times \Pone} \Phi _{\hat A_k} \ d\hat{m}_k < \lambda _+ (\hat{A},\hat{\mu})
\end{displaymath}
hence  $\hat{m} \neq \hat{m}^u$.  It follows that  $\alpha \neq 0$ and 
$$\hat{m}^s = \tfrac{1}{\alpha}\left(\hat m- \beta \hat{m}^u\right)$$
is a $u$-state for $H^{u,\hat{A}}$.  Similarly, $\hat{m}^u$ is an $s$-state for $H^{s,\hat{A}}$.  In particular, $\hat m^s$ and $\hat m^u$ are $su$-states.  Claim \ref{lemma:convex combination} completes the proof.
\end{proof}

\subsection{Final reductions and standing notation.  }

As  discussed in the proof of Lemma \ref{states}, the family of invariant stable holonomies defines a continuous change of linear coordinates on the fibers $\lbrace x\rbrace\times \Pone$ that makes the cocycle constant along local stable manifolds of $\hat{f}$.  The convergence of the cocycles $\hat{A}_k\to \hat{A}$ is not affected by this coordinate change.  Moreover,  the straightened out cocycles  admit unstable holonomies with the appropriate convergence and have the same Lyapunov exponents.
We assume for the remainder we have straightened out the cocycles in \eqref{eq:mainassump} along their respective stable holonomies.   Following the notation introduced in Section \ref{sec: continuity conditional measures}, let $A, A_k:\Sigma \rightarrow SL(2,\mathbb{R})$ be such that $\hat{A}=A\circ P$ and $\hat{A}_k=A_k\circ P$ where $P\colon \hat \Sigma \to \Sigma $ is the natural projection. 

We assume for the remainder that 
$$\lambda_+(\hat A_k,\hat\mu_k) \not \to \lambda_+(\hat A,\hat\mu)$$
and fix a sequence of ergodic $u$-states $\hat{m}_k$  as in the proof of Proposition \ref{prop: characterization of disc. points}.  We  
assume $\hat{m}_k$ converges to some measure $\hat{m}$.    
From (the proof of) Proposition \ref{prop: characterization of disc. points}, we have that $\hat{m}=\alpha \hat{m}^s + \beta \hat{m}^u$ and, moreover that $\hat{m}^s$  and $\hat{m}^u$ are $su$-states.

 From Proposition \ref{Prop: continuous $su$-states}, it follows that there are continuous functions $\sigma^{s/u}\colon \hat \Sigma \to \Pone$ such that $E^{s/u,\hat{A}}_{\hat{x}} = \sigma^{s/u}(\hat{x}).$  
Using $\sigma^{s/u}$, we perform a final continuous change of coordinates, that is projective in each fiber, such that for $\hat x\in \hat \Sigma$ $$\sigma^{s}(\hat x ) = [1:0]:= q, \text{ and }\sigma^{u}(\hat x ) = [0:1]:=p.$$  
In particular, after this coordinate change the projective 
 cocycle $\P \hat A(y)$ leaves $q$ and $p$ invariant for every $y$.  Note that the change of coordinate is constant on local stable manifolds so the cocycle $\hat A$ is still of the form $\hat A = A \circ P$ for some $A\colon \Sigma\to \Sl(2,\R)$.  Note that in order to define this coordinate change, we heavily use that the limiting measure $\mu$ is fully supported.  
 
 We take $m_k := (P\times Id)_{\ast}\hat{m}_k$ and similarly take 
$m := (P\times Id)_{\ast}\hat{m}$, $m^s := (P\times Id)_{\ast}\hat{m}^s$, $m^u := (P\times Id)_{\ast}\hat{m}^u$.  Each of the above measures is 
 induced by a $u$-state on $\hat \Sigma\times \Pone$ and hence induces a continuous family of conditional measures.  Since the measures $\hat{m}^{u}_{k}$ are ergodic for each $k$ we conclude that the projected measures $m_{k}$ are ergodic.

Let $\{m^k_x\}$ and $\{m_x\}$ denote a continuous family of conditional measure for $m_k$ and $m$, respectively, given by Proposition \ref{prop: continuity of A mu invariant}. Observe that, for every $x\in \Sigma$, $m_x=\alpha \delta _q +\beta \delta _p$ where $\alpha ,\beta \in (0,1)$.  
We split the proof of Theorem \ref{mainthm} into two cases.  In Section \ref{sec:atomic case} we consider the case that for infinitely many $k$ there is a $x\in \Sigma$ such that 
the conditional measure $m^k_x$ has an  atom.  In Section \ref{sec:non-atomic case}
 we consider the case that the measures $m^k_x$ are non-atomic for every $x$ and infinitely many $k$.   Passing to subsequences, we can assume that either  the measures $m^k_x$ are non-atomic for all $x$ and $k$ or contains an atom for some $x$ and  all $k$. 
 In both cases, we derive a contradiction showing that $(\hat A, \hat \mu)$ can not satisfy
 \eqref{eq:mainassump}.

\section{Case 1: the measures $m^k_x$  are atomic} \label{sec:atomic case} 

In this section we will deduce a contradiction  to \eqref{eq:mainassump} under the assumption that  for every $k\in \N$ there is some $x\in \Sigma$ such that the conditional measure $m^k_x$ contains an atom.  We first claim that $m^k_x$  contains an atom for \emph{every} $x\in \Sigma$ which by ergodicity, implies that the measures $m^k_x$  are all finitely supported.  The proofs of Lemmas \ref{lemma: atoms on the projections} and \ref{lemma: atoms on the conditional measures} given below are not new; to the best of our knowledge they first appear as a consequence of \cite[Lemmas 5.2, 5.3]{BV}. We reproduce the proofs here for completeness. 

For each $k$, consider
\begin{displaymath}
\gamma ^k_0:=\sup \left\lbrace m^k_x(v) :  x\in \Sigma, v\in \Pone\right\rbrace .
\end{displaymath}
By hypothesis, $\gamma ^k_0>0$ for all $k$.  

\begin{lemma}\label{lemma: atoms on the projections}
For each $x\in \Sigma$ there exists $v^k_x\in \Pone$ such that $m^k_x(v^k_x)=\gamma ^k_0$. Moreover, $m^k_x(v^k_x)=\gamma ^k_0$ if and only if $m^k_y(A_k(y)^{-1}(v^k_x))=\gamma ^k_0$ for all $y\in f^{-1}(x)$.
\end{lemma}
\begin{proof}
Consider $\Gamma ^k_0:=\lbrace x\in \Sigma : m^k_x(v)=\gamma ^k_0 \; \textrm{for some} \; v\in \Pone\rbrace$. We argue that this is a non-empty closed set. Indeed, let $\lbrace x_j \rbrace _{j\in \mathbb{N}} \subset \Sigma$ and $\lbrace v_j\rbrace _{j\in \mathbb{N}}\subset \Pone$ be sequences such that $m^k_{x_j}(v_j)\xrightarrow{j\rightarrow \infty} \gamma ^k_0$. Restricting to a subsequence we may assume that $\lbrace x_j \rbrace _{j\in \mathbb{N}}$ converges to some $x\in \Sigma$ and $\lbrace v_j\rbrace _{j\in \mathbb{N}}$ converges to some $v\in \Pone$. Now, as $x\mapsto m^k_x$ is continuous, for each $\varepsilon >0$ we have that
\begin{displaymath}
\gamma ^k_0\leq \limsup _{j\rightarrow \infty} m^k_{x_j}(\overline{B(v,\varepsilon)})\leq m^k_x(\overline{B(v,\varepsilon)}).
\end{displaymath}
Thus   $m^k_x(v)\le  \gamma ^k_0$  and hence $m^k_x(v)=  \gamma ^k_0$. It follows that  $\Gamma ^k_0$ is non-empty and closed.

By Remark \ref{remark: invariant measure},
$$m^k_x(v)=\sum _{y\in f^{-1}(x)} \frac{1}{J_{\mu_k}f(y)}m^k_y(A_k(y)^{-1}(v))$$
for all $x\in \Sigma$ and $v\in \Pone$. As $\sum _{y\in f^{-1}(x)} \frac{1}{J_{\mu_k}f(y)}=1$ for every $x\in \Sigma$, it follows that $m^k_x(v)=\gamma ^k_0 $ if and only if $m_y(A_k(y)^{-1}(v))=\gamma ^k_0$ for every $y\in f^{-1}(x)$.  In particular, $f^{-1}( \Gamma ^k_0)\subset \Gamma^k_0$. Since $f$ is transitive, $\Sigma$ is the unique  non-empty, closed, backwards-invariant subset of $\Sigma$.   Hence  $\Gamma ^k_0=\Sigma$. 
\end{proof}

We show that points realizing the maximal atomic mass of $m^k_x$ have the same property for the  measure $\hat m_k$.    
\begin{lemma}\label{lemma: atoms on the conditional measures}
Given $x\in \Sigma$ and $v\in \Pone$ we have that $\hat{m}^k_{\hat{x}}(v)\leq \gamma ^k_0$ for $\hat{\mu}^k_x$ almost every $\hat{x}\in W^s_{\loc}(x)$. Consequently, $m^k_x(v^k_x)=\gamma ^k_0 $ if and only if $ \hat{m}^k_{\hat{x}}(v^k_x)=\gamma ^k_0$ for $\hat{\mu}^k_x$ almost every $\hat{x}\in W^s_{\loc}(x)$. 
\end{lemma}

\begin{proof}
Suppose that there exist $v\in \Pone$, $x\in \Sigma$, $\gamma _1 >\gamma ^k_0$ and a subset $Z\subset W^s_{\loc}(x)$ with positive $\hat{\mu}^k_x$-measure  such that $\hat{m}^k_{\hat{x}}(v)\geq \gamma _1$ for every $\hat{x}\in Z$. For any $n\geq 0$ let us consider the partition of $W^s_{\loc}(x)$ given by
\begin{displaymath}
\lbrace \hat{f}^n(W^s_{\loc}(y)):  y\in f^{-n}(x)\rbrace.
\end{displaymath}
Observe that the diameter of this partition goes to zero when $n$ goes to infinity.
Therefore, by the regularity of the measures $\hat{\mu}^k_x$, given $\varepsilon >0$ we can find $n\geq 1$ and $y\in f^{-n}(x)$ such that 
\begin{equation}\label{eq: measure of Z intersected with elements of the partition}
\hat{\mu}^k_x(Z\cap \hat{f}^n(W^s_{loc}(y))) > (1-\varepsilon )\hat{\mu}^k_x(\hat{f}^n(W^s_{loc}(y))).
\end{equation}
Indeed, take a closed  $F\subset W^s_{loc}(y)$ and an open  $A\subset W^s_{loc}(y)$ with $F\subset Z \subset A$ and $\hat{\mu}^k_x(F)>(1-\varepsilon)\hat{\mu}^k_x (A)$. 
Choosing $n$ sufficiently large such that the diameter elements of the partition $\lbrace \hat{f}^n(W^s_{loc}(y)):  y\in f^{-n}(x)\rbrace$ are smaller than $\textrm{dist} (F, A^c)>0$, we have  for every $y\in f^{-n}(x)$ such that $ \hat{f}^n(W^s_{loc}(y)) \cap F \neq \emptyset$ that  $ \hat{f}^n(W^s_{loc}(y)) \subset A$. If \eqref{eq: measure of Z intersected with elements of the partition} fails  then, with $G= \{  y\in f^{-n}(x) : \hat{f}^n(W^s_{loc}(y)) \cap F \neq \emptyset\}$ we have 
\begin{displaymath}
\begin{split}
\hat{\mu}^k_x(F)&=  \sum _{y\in G}\hat{\mu}^k_x ( \hat{f}^n(W^s_{loc}(y)) \cap F ) \leq \sum_{y\in G} \hat{\mu}^k_x ( \hat{f}^n(W^s_{loc}(y)) \cap Z ) \\
&\leq  (1-\varepsilon) \sum _{y\in G}\hat{\mu}^k_x ( \hat{f}^n(W^s_{loc}(y)) ) \leq (1-\varepsilon)\hat{\mu}^k_x (A).
\end{split}
\end{displaymath}

Take $\varepsilon >0$ so  that $(1-\varepsilon)\gamma _1 >\gamma ^k_0$. As  $\hat{m}_k$ is an $F_{\hat{A}_k}$-invariant measure we have
\begin{displaymath}
\hat{m}^k_{\hat{y}}(A^n_k(y)^{-1}(v))=A^n_k(y)_{\ast}\hat{m}^k_{\hat{y}}(v)=\hat{A}^n_k(\hat{y})_{\ast}\hat{m}^k_{\hat{y}}(v)=\hat{m}^k_{\hat{f}^n(\hat{y})}(v)\geq \gamma _1
\end{displaymath}
for almost every $\hat{y}\in \hat{f}^{-n}(Z)\cap W^s_{\loc}(y)$. By Remark \ref{remark: property of mu},
\begin{eqnarray*}
\hat{\mu}^k_y(\hat{f}^{-n}(Z)\cap W^s_{\loc}(y))&=& J_{\mu _k}f^n(y)\hat{\mu}^k_x(Z\cap \hat{f}^n(W^s_{\loc}(y)))\\
&\geq & (1-\varepsilon)J_{\mu _k}f^n(y)\hat{\mu}^k_x(\hat{f}^n(W^s_{\loc}(y)))=(1-\varepsilon)
\end{eqnarray*}
and it follows that
\begin{displaymath}
m^k_y(A^n_k(y)^{-1}(v))=\int _{W^s_{\loc}(y)} \hat{m}^k_{\hat{y}}(A^n_k(y)^{-1}(v))d\hat{\mu}^k_y(\hat{y})\geq (1-\varepsilon)\gamma _1 >\gamma ^k_0
\end{displaymath}
contradicting the definition of $\gamma ^k_0$.  
\end{proof}

Note that we have shown that if $m^k_x$ contains an atom for some $x\in \Sigma$, then $m^k_x$ contains an atom of mass $\gamma_0^k$ for every $x\in \Sigma$.  By ergodicity of the measure $m_k$, it follows that $m^k_x$ is finitely supported for every $x\in \Sigma$ and, moreover, that every atom of $m^k_x$ has mass $\gamma_0^k$.  
To derive a contradiction to  \eqref{eq:mainassump}, we further divide the case that $m^k_x$ contains atoms into two subcases.

\subsection{Case A: Positive Lyapunov Exponents}\label{subsec: positive lyap exponent atomic case} Passing to a subsequence, assume that 
 $\lambda _+(\hat{A}_{k},\hat{\mu}_k)>0$ for every $k\in \mathbb{N}$. 

Recall Lemma \ref{lemma: atoms on the projections}.  Given $x\in \Sigma$, let $v^k_x\in \Pone$ be such that such that $m^k_x(v^k_x)=\gamma ^k_0 >0$ for all $x\in \Sigma$. From Lemma \ref{lemma: atoms on the conditional measures} we have   $\hat{m}^k_{\hat{x}}(v^k_x)=\gamma ^k_0$ for $\hat{\mu}_x$ almost every $\hat{x}\in W^s_{\loc}(x)$. But, as we  assume   $\lambda _+(\hat{A}_k,\hat{\mu}_k)>0$, it follows from the definition of  $\hat{m}_k$ (see the proof of Proposition \ref{prop: characterization of disc. points}) that $\hat{m}^k_{\hat{x}}=\delta _{E^{u,k}_{\hat{x}}}$. Consequently, $\gamma ^k_0 =1$ and $v^k_x=E^{u,k}_{\hat{x}}$ for $\hat{\mu}_x$ almost every $\hat{x}\in W^s_{\loc}(x)$.  

It then follows  from 
Lemma \ref{lemma:convergence of mkx to mx} that $m^k_x=\delta _{v^k_x}$ converges to $m_x=\alpha \delta _{q} +\beta \delta _{p}$ for every $x\in \Sigma$.  Since $\alpha ,\beta \in (0,1)$ and $p\neq q$ this gives  a contradiction.

\subsection{Case B: Zero Lyapunov Exponents} \label{subsec: zero lya exponent atomic case} We now suppose $\lambda _+(\hat{A}_{k},\hat{\mu}_k)=0$ for every $k\in \mathbb{N}$.

First note that, as $\lambda _+(\hat{A}_{k},\hat{\mu}_k)=0=\lambda _-(\hat{A}_{k},\hat{\mu}_k)$ for every $k\in \mathbb{N}$, by Theorem \ref{theorem: invariance principle}, each measure $\hat{m}_k$ is an $su$-state.  By Proposition \ref{Prop: continuous $su$-states} we may find an $su$-invariant disintegration into a continuous family of  conditional measures $\{\hat{m}^k_{\hat{x}}\}_{\hat{x}\in \hat{\Sigma}}$. As the stable holonomies are trivial, we have that  $m^k_x= \hat{m}^k_{\hat{x}}=\hat{m}^k_{\hat{z}}$ for every $\hat{x}, \hat{z}\in W^s_{\loc}(x)$ and all $x\in \Sigma$. By Proposition \ref{prop: characterization of disc. points} the same property holds for the disintegrations of  $\hat{m}$ and $m$.  

In particular this allows to identify $m_x^k$ and $m_y^k$ via unstable holonomies, for $x,y$ in the cylinder $[0;i]$ for $1 \leq i \leq \ell$. 
\begin{claim}\label{lemma: my is a push-forward of mx}
For each $1 \leq i \leq \ell$ and every $x,y \in [0;i]$,  $\hat{x}\in W^s_{\loc}(x)$, and $\hat y =W^s_{\loc}(y)\cap W^u_{\loc}(\hat{x})$ we have 
\begin{displaymath}
m_y=\left( H^{u,\hat{A}}_{\hat{x}\hat{y}}\right)_{\ast} m_x \quad \textrm{and} \quad m^k_y=\left( H^{u,k}_{\hat{x}\hat{y}}\right)_{\ast} m^k_x .
\end{displaymath}
\end{claim}

Let $\lbrace V^k_x \rbrace _{x\in \Sigma}$ be the family of finite subsets of $ \Pone$ given by 
\begin{displaymath}
V^k_x:= \lbrace v\in  \Pone:\; m^k_x(v)=\gamma ^k_0 \rbrace .
\end{displaymath}
Note that $m_x^k (V^k_x) = 1$. Moreover, combining Lemma \ref{lemma: atoms on the projections} and the previous claim we have

\begin{claim}\label{lemma: invariance of Vkx}
For $x,y\in \Sigma$, and $k\in \mathbb{N}$,
\begin{enumerate}
\item $\textrm{card}(V^k_x)=\textrm{card}(V^k_y)$,  
\item \label{item:2}$A_k(x)(V^k_x)=V^k_{f(x)}$, and 
\item \label{item:3} $V^k_y=H^{u,k}_{\hat{x}\hat{y}}(V^k_x)$ for any  $\hat{x}\in W^s_{\loc}(x)$ and $\hat{y}\in W^s_{\loc}(y)\cap W^u_{\loc}(\hat{x})$.
\end{enumerate}
\end{claim}

We now bound the number of  atoms appearing in the measure $m_x^k$.

\begin{lemma}\label{lemma: cardinality of big atoms}
For every $x\in \Sigma$ we have that
$ \textrm{card}(V^k_x)\leq 2$
for $k$ sufficiently large.
\end{lemma}

\begin{proof}
As $ \textrm{card}(V^k_x)$ is defined for every $x$ and is moreover constant, it is enough to prove that $\textrm{card}(V^k_x)\leq 2$ for some $x\in \Sigma$.

We claim there is a periodic point $x\in \Sigma$ with period $\ell$ such that $ A^\ell(x):= A(f^{\ell-1}(x)) \dots A(x)$ is hyperbolic.  Indeed, recall that the cocycle $A(x)$ preserves the coordinate axes and is thus of the form $$A(x) =\left(\begin{array}{cc}\gamma(x) & 0 \\0 & \gamma(x)\inv\end{array}\right).$$
If follows that the logarithm of the eigenvalues of $A^\ell(x)$ for any such periodic point $x$ are 
	$$\frac{1}{\ell}\sum_{j= 0}^{\ell-1} \log \gamma ^{\pm 1} (f^j(x)).$$ 
If the logarithm of the eigenvalues of $A^\ell(x)$ vanished for every periodic point $x$ then, as measures concentrated on periodic orbits are dense in the set of all $f$-invariant measures, it follows that $\int \log \gamma(x) \ d \mu'(x)=0$ for every $f$-invariant measure $\mu'$.  It follows that the Lyapunov exponents of the cocycle vanish for every $f$-invariant measure $\mu'$ contradicting our assumption on the measure $\mu$.
The matrix $A^\ell(x)$ is thus hyperbolic for some periodic point $x$ and, as the set of hyperbolic matrices is open, for $k$ sufficiently large $A^\ell_k(x)$ is also hyperbolic. Therefore, as $A^\ell_k(x)(V^k_x)=V^k_x$ and $V^k_x$ is finite, it follows that $\textrm{card}(V^k_x)\leq 2$.  
\end{proof}

Let $V^k_x= \{ v^k_x\}$ or $V^k_x= \{ v^k_x, w^k_x\}$  depending on the cardinality of $V^k_x$.  
As 
 $m_k$ is ergodic,   either
$$m_k = \begin{cases}
	\displaystyle \int \delta_{(x, v^k_x)} \ d \mu_k(x) &  \textrm{card}(V^k_x)=1 \\ 
	\displaystyle\int \tfrac{1}2 \delta_{(x, v^k_x)} +  \tfrac{1}2 \delta_{(x, w^k_x)} \ d \mu_k(x) &  \textrm{card}(V^k_x)=2 .
	\end{cases}$$
	
As before, we write $\Phi_{A_k}\colon \Sigma\times \Pone\to \R$ for
$$\Phi _{A_k} (x,v) =  \log\left(\dfrac{\| A_k(x)(v)\|}{\| v\|}\right).$$
Recalling that the cocycle $\hat A_k$ is constant along local stable manifolds and recalling the definition of the measure $m_k$ we have 
\begin{equation}
\label{lemma: mk representation}
\lambda _+(\hat{A}_k,\hat{\mu}_k)=\int _{\hat \Sigma\times \Pone} \Phi _{\hat A_k} d\hat m_k =\int _{\Sigma\times \Pone} \Phi _{A_k} dm_k.\end{equation} 
In particular 
\begin{equation}\label{eq: exponents as the integral over the family vx wx}
0= \lambda _+(\hat{A}_k,\hat{\mu}_k)= 
\begin{cases}\displaystyle \int  \Phi_{A_k} (x, v_x^k) \ d\mu_k (x) &  \textrm{card}(V^k_x)=1\\
\displaystyle\int  \tfrac{1}2 \Phi_{A_k} (x, v_x^k) + \tfrac 1 2  \Phi_{A_k} (x, w_x^k) \ d\mu_k (x)
&  \textrm{card}(V^k_x)=2.\end{cases}
\end{equation}

We now consider two subcases depending on the cardinality of $V^k_x$.  
\subsubsection{$\textrm{card}(V^k_x)=1$ for every $k\in \mathbb{N}$} \label{subsec: card equals to 1}
Passing to a subsequence, suppose that $\textrm{card}(V^k_x) =1$ for every $k\in \mathbb{N}$. For every $x\in \Sigma$, let $V^k_x=\lbrace v^k_x\rbrace$.

Fix $x\in \Sigma$. Up to restricting to a subsequence, we may assume that $v^k_x$ converges to some $v_0$ in $\Pone$. Thus, as $H^{u,k}_{\hat{x}\hat{y}}$ converges to $H^{u,\hat{A}}_{\hat{x}\hat{y}}$ from Claim \ref{lemma: invariance of Vkx}(\ref{item:3}) for $y\in \Sigma$ 
 in the same cylinder as $x$ 
 we have (fixing any $\hat x\in W^s_{\loc}(x)$ and  $\hat{y}\in W^s_{\loc}(y)\cap W^u_{\loc}(\hat{x})$) that 
\begin{displaymath}
v^k_{y}=H^{u,\hat A_k}_{\hat{x}\hat{y}}(v^k_x) \xrightarrow{k\rightarrow \infty} H^{u,\hat{A}}_{\hat{x}\hat{y}}(v_0):=v_y.
\end{displaymath}

This combined with Claim \ref{lemma: invariance of Vkx}(\ref{item:2}) and  \ref{lemma: invariance of Vkx}(\ref{item:3}) implies that $v^k_y$ converges to some $v_y$ in $\Pone$ for every $y\in \Sigma$.  Moreover, 
the family $\left\lbrace v_y\right\rbrace _{y\in \Sigma}$ satisfies
\begin{equation}\label{eq: invariance of v_y by A and holonoimes card 1}
A(y)(v_y)=v_{f(y)} \quad \textrm{and} \quad H^{u,\hat{A}}_{\hat{y}\hat{z}}(v_y)=v_z.
\end{equation}
and  the map $y\mapsto v_y$ is continuous.  
As $A(y) v_y = v_{f(y)}$,  the graph of $y\mapsto v_y$ is a closed, $F_A$-invariant subset of $\Sigma\times \P ^1$.  Hence (by an argument similar to the proof of  Claim \ref{lemma:convex combination}),
 the (non-uniform) hyperbolicity of the cocycle $A$ implies that either $v_y = q$ for every $y\in \Sigma$ or $v_y= p$ for every $y$ .  

Suppose first that  $v_0=q$. Then  from \eqref{eq: invariance of v_y by A and holonoimes card 1}
\begin{equation}\label{eq:star}
0=  \int   \Phi_{A_k}(x,  v^k_x )\ d\mu_k  (x) \to \int _{\Sigma} \Phi_{A}(x,  q)\ d\mu (x) = 
 \int \Phi_{ A} \ d  m^s 
 = \lambda _-(\hat{A},\hat{\mu})
\end{equation}
which is a contradiction since  $\lambda _-({\hat{A}},\hat{\mu})<0$. 
Similarly, if  $v_y = p$ for every $y$ then 
 \begin{equation}\label{eq:star2}
0=  \int  \Phi_{A_k}(x,  v^k_x )\ d\mu_k  (x) \to 
\int  \Phi_{A}(x,  p) \ d\mu  (x) = 
\int \Phi_A \ d  m^u 
 = \lambda _+(\hat{A},\hat{\mu})
\end{equation} contradicting  that  $\lambda _+({\hat{A}}, \hat{\mu})> 0$.

\subsubsection{$\textrm{card}(V^k_x) =2$ for every $k\in \mathbb{N}$} \label{subsec: card equals to 2}
If $\textrm{card}(V^k_x) =2$ for every $k $ we may take  $V^k_y=\lbrace v^k_y, w^k_y\rbrace$ with 
\begin{equation}\label{eq: vkx and wkx are related with vky and wky by holonomies}
v^k_y=H^{u,k}_{\hat{x}\hat{y}}(v^k_x) \quad \textrm{and} \quad w^k_y=H^{u,k}_{\hat{x}\hat{y}}(w^k_x)
\end{equation}
for every $x$ and $y$
in the same cylinder. Moreover,
\begin{equation}\label{eq: invariance of the family vx and wx}
A_k(y)\left( \left\lbrace v^k_y, w^k_y \right\rbrace \right)= \left\lbrace v^k_{f(y)}, w^k_{f(y)}\right\rbrace
\end{equation}
for every $y\in \Sigma$. 

Fix $x\in \Sigma$. Passing to  subsequences suppose that $v^k_x$ converges to $v_0$ and $w^k_x$ converges to $w_0$ in $\Pone$. Then, by \eqref{eq: vkx and wkx are related with vky and wky by holonomies}, it follows that
\begin{displaymath}
v^k_y= H^{u,k}_{\hat{x}\hat{y}}(v^k_x) \xrightarrow{k\rightarrow \infty} H^{u,\hat{A}}_{\hat{x}\hat{y}}(v_0):=v_y 
\end{displaymath}
and
\begin{displaymath}
w^k_y= H^{u,k}_{\hat{x}\hat{y}}(w^k_x) \xrightarrow{k\rightarrow \infty} H^{u,\hat{A}}_{\hat{x}\hat{y}}(w_0):=w_y 
\end{displaymath}
for every $y$ in the same cylinder as $x$. Invoking \eqref{eq: invariance of the family vx and wx} and \eqref{eq: vkx and wkx are related with vky and wky by holonomies} it follows that   $v^k_y$ converges to some $v_y$ and $w^k_y$ converges to some $w_y$ in $\Pone$ for \emph{every} $y\in \Sigma$.  Moreover,  $y\mapsto v_y$ and $y\mapsto w_y$ are  continuous and
\begin{equation}\label{eq: invariance of v_y and w_y by A card 2}
A(y)\left( \left\lbrace v_y,w_y\right\rbrace\right)=\left\lbrace v_{f(y)},w_{f(y)}\right\rbrace
\end{equation}
and
\begin{equation}\label{eq: invariance of v_y and w_y by holonomies card 2}
v_z= H^{u,\hat{A}}_{\hat{y}\hat{z}}(v_y) \quad \textrm{and} \quad w_z= H^{u,\hat{A}}_{\hat{y}\hat{z}}(w_y)
\end{equation}
for every $y,z\in \Sigma$ in the same cylinder.

Suppose that $v_0=w_0$.  Then as argued above, either $v_y= w_y=p$ for all $y\in \Sigma$ or $v_y= w_y=q$ for all $y\in \Sigma$ and from \eqref{eq: exponents as the integral over the family vx wx}   we arrive at the same contradictions as in \eqref{eq:star} and \eqref{eq:star2}  in  the previous case.

If $v_0\neq w_0$ then (by ergodicity) without loss of generality we may assume $v_y = q$ and $w_y= p$ for all $y\in \Sigma$.  However, as we assumed  $m_k$ to be ergodic, we have that $A_k(y) v_y^k = w_{f(y)}^k$ for some  $y\in \Sigma$.  Then $A(y) v_y = w_{f(y)}$ and hence  $A(y)(p) = q$, a contradiction.

\section{Case 2: the measures \textbf{$m^k_x$} are non-atomic} \label{sec:non-atomic case}

We now derive a contradiction to \eqref{eq:mainassump} under the assumption that the measures  $m^k_x$ are non-atomic for every $x\in \Sigma$ and every $k$.  
To arrive at a contradiction, we introduce the tools of couplings and (additive) energies: $\lbrace m^k_x \rbrace _{x\in \Sigma}$ being non-atomic implies that the family of measures $\lbrace \restrict{m^k_x}{U_x}\rbrace $ obtained from $\lbrace m^k_x \rbrace$ by restriction to a suitable family of sets $\lbrace U_x \rbrace$ admit a family of symmetric self-couplings with finite energy. Taking advantage of the fact that the stable space is a repeller for the action of the cocycle $A$ on $\Pone$, we are able to build a new family of symmetric self-couplings of $\lbrace \restrict{m^k_x}{U_x}\rbrace _{x\in \Sigma}$ with energy strictly smaller by a definite factor coming from the rate of expansion of $A$ at the stable space. We can then iterate this procedure to construct a symmetric self-coupling of  $\lbrace \restrict{m^k_x}{U_x}\rbrace $ with negative energy, arriving at a contradiction. This approach follows the main ideas in \cite[Chapter 10]{V2} though we deviate slightly from  \cite{V2} by using  \emph{additive Margulis functions}  introduced in \cite{AEV}.  The contradiction is given by Proposition \ref{prop:key}.

Recall we changed  coordinates on the cocycle $A$ so that $A(y)$ fixes  $q=[1:0]$ and $p= [0:1]$ for every $y\in \Sigma$. Recall the sequence of measures $m_k$ projecting to $\mu_k$ and admitting continuous family of conditional measures $\{m^k_x\}_{x\in \Sigma}$.  Moreover $m_k$ converges to the measure $\mu\times (\alpha \delta_{q} + \beta \delta_{p})$ and the conditional measures $m^k_x$ converge uniformly to $ \alpha\delta_{q} + \beta \delta_{p}$.
{Recall we assume $\alpha>0$.}  Also, recall we write $\P A\colon \Sigma\to \mathrm{Diff^\infty}(\Pone)$, $x\mapsto \P A(x)$ for the projective cocycle.  Similarly, write $\P A_k$ for the projectivized cocycle of $A_k$.  

\subsection{$q$ is an expanding point}

We begin by recalling that, given $B\in GL(2,\mathbb{R})$ and $v\in \Pone$, the derivative at the point $v$ of the action of $\P B$ in the projective space is given explicitly by 
\begin{equation*}
D_v\P B (\dot{v})=\frac{\textrm{proj} _{B(v)}B(\dot{v})}{\parallel B(v)\parallel }\quad \textrm{for every} \: \dot{v}\in T_{v}\Pone=\lbrace v\rbrace ^{\perp}
\end{equation*}
where $\textrm{proj}_v: w\rightarrow w-v\frac{<w,v>}{<v,v>} $ denotes the orthogonal projection to the hyperplane orthogonal to $v$.

\begin{claim}
For almost every $x\in \Sigma$ we have 
$$\lim_{n\to \infty} \log\left(\left\| D_{q} \P A^n (x) \right \| ^{1/n}\right )= \lambda_+(\hat A, \hat \mu) - \lambda_-(\hat A, \hat \mu) =:c >0.$$
\end{claim}

\begin{proof}
Recall $p,q$ are orthogonal and preserved by the cocycle $A$.  

Then for $v\in T_{q} \Pone =\{q\}^{\perp}$ we have $v\in \mathrm{span}{(p)}$.  Thus if  $\|v\|=1$ we have $\|A^n(y)(v)\| = \|A^n(y)(p)\|$.  Projecting back to $T_{q}\Pone$ we have 
$$\|D_{q} \P A^n(y)(v)\| = \dfrac{ \|A^n(y)(p)\|}{\|A^n(y)(q)\|}.$$
The claim then follows from the pointwise ergodic theorem.  
\end{proof}

\begin{claim}We may select $N\in \mathbb{N}$ such that
\begin{displaymath}
\int _{ \Sigma} \log \left(\| D_{q} \mathbb{P}A^N(x)\|\right) d\mu (x)>6.
\end{displaymath}
\end{claim}
\begin{proof}
We have $\lim_{n\to \infty} \frac 1 n \log\left(\| D_{q} \P A^n(x) \| \right)\to c>0$ almost everywhere. Moreover, as $\frac 1 n \log\left(\| D_{q} \P A^n(x)  \|\right) $ is bounded above and below uniformly in $x$ and $n$, by dominated convergence   we have	
$$\lim_{n\to \infty} \frac 1 n \int \log\left(\| D_{q} \P A^n(x) \|\right) d \mu(x) \to c>0.\qedhere $$ \end{proof}

Fix such an $N$ for the remainder.  
We define 
$$
{\kappa (x):= \log \left(\| D_{q} \mathbb{P}A^N(x)\|\right).} 
$$
As $\kappa\colon \Sigma\to \R$ is a continuous function, for all sufficiently large $k$ we have \begin{equation}\label{eq:kappa} \int \kappa (x) \ d \mu_k (x) > 4.\end{equation}

\subsection{Couplings and energy} Let $d$ be the distance on $\Pone$ defined by the angle between two directions.    We assume $d$ is normalized so that $\Pone$ has diameter 1.  

Consider  a Borel probability measure $\mu'$ on $\Sigma$ and a $\mu'$-measurable  family $\{\nu_x\}_{x\in \Sigma}$ of  finite
Borel measures on $\Pone$.  The measures $\nu_x$ are not assumed to be probabilities nor are they assumed to have the same mass.  For $j\in \{1,2\}$, let $\pi _j :\Pone\times \Pone\to  \Pone$ be the projection on the $j$-th factor.  For $x\in \Sigma$, let $\xi_x$ be a measure on $\Pone \times \Pone$. We say a parameterized family of measures $\{\xi_x\}_{x\in \Sigma}$ on $\Pone \times \Pone$ is a (measurable) \emph{family of   symmetric  self-couplings} of  $\{\nu_x\}_{x\in \Sigma}$ if 
 \begin{enumerate}
\item   $x\mapsto \xi_x$ is $\mu'$-measurable,
  \item $(\pi_j)_*\xi_x = \nu_x$ for $j\in\{1,2\}$,  and  
  \item $\iota_*\xi_x = \xi_x$ where $\iota\colon \Pone\times \Pone \to \Pone \times \Pone$ is the involution $\iota\colon (u,v)\rightarrow (v,u)$.
  \end{enumerate} 
 We note that we always have one family of symmetric self-couplings constructed by taking for every $x$ the measure
 $$\xi_x =\frac{1}{\| \nu_x \|}\nu _x \times \nu _x$$ for all $x$ with $\|\nu_x\|\neq 0$ 
 where $\| \nu_x \|:= \nu_x (\Pone)$ denotes the mass of the measure $\nu_x $.

We define a function $\varphi \colon \Pone \times \Pone \to \R$ by 
\begin{equation}\label{eq:phidef}
\phi (u,v)=- \log d(u,v).
\end{equation}
Note that $\phi$ is non-negative. In the language of \cite{AEV}, the function  $\phi$ is an \emph{additive Margulis function} and its  properties will be used to deduce the contradiction in Proposition \ref{prop:key} below. For    a {family of   symmetric  self-couplings} $\{\xi_x\}_{x\in \Sigma}$ of  $\{\nu_x\}_{x\in \Sigma}$ as above we define the (additive) \textit{energy} of  $\{\xi_x\}_{x\in \Sigma}$ to be
\begin{displaymath}
\int _{\Sigma} \int_{ \Pone\times \Pone} \phi  (u,v) \  d\xi_x(u,v)  d \mu'(x).
\end{displaymath}

\subsection{Choice of parameters}  
To establish a contradiction to \eqref{eq:mainassump} we select a number of parameters that will be fixed for the remainder.   Recall the $N$ fixed above and the function $\kappa$.  
\begin{enumerate}
\item Let $U_0 \subset \Pone$ be an open ball centered at $q$ with  $p\not \in \overline{U}_0$. 
\item Let $U_1\subset \overline{U_1}\subset U_0$ be an open neighborhood of $q$ such that for every $x\in \Sigma$ and every sufficiently large $k$ we have 
\begin{enumerate}
\item  $\P A_k^N(x)(\overline{U_1})\subset U_0$; 
\item $\overline{U_1}\subset \P A_k^N(x)(U_0)$;
\item  for every $u,v\in U_1$ 
\begin{equation} \label{eq: estimative 1}
d(\P A^N_k(x)(u),\P A^N_k(x)(v))\geq  e^{-\alpha} e ^{\kappa (x)}d(u,v).
\end{equation}
\end{enumerate}
\end{enumerate}
From \eqref{eq: estimative 1} it follows  that for every $u,v\in U_1$, $x\in  \Sigma$ and $k$ sufficiently large
\begin{equation}\label{eq: estimative 2}
\varphi(\P A^N_k(x)(u),\P A^N_k(x)(v))\leq  \varphi(u,v) -\kappa (x)  + \alpha.
\end{equation}

\begin{enumerate}[resume]
\item Fix $ q\in U_4\subset \overline{U_4}\subset U_3\subset \overline{U_3}\subset U_2\subset \overline{U_2}\subset U_1$ such that
\begin{enumerate}
\item  each $U_j$ is an open set;
\item  $\P A^N_k(y)(\overline{U}_4)\subset U_3$ for every $y\in \Sigma$ and $k$ sufficiently large.
\end{enumerate}

\item By compactness of $\Sigma$ and uniform convergence of $A_k$ to $A$, we may select $M_1>1$ so that for all $x\in \Sigma, u\in \Pone$ and $k$ sufficiently large,
$$-M_1 < \log(\|D_u\P A_k^N(x)\| )< M_1.$$
Note in particular that $|\kappa(x)|\le M_1$.  
\item Take $M_2>1$ to be the maximum of  
$$\sup \lbrace \varphi (u,v): u\in U_3, v\in U_2^c\rbrace \text{ and } 
\sup \lbrace \varphi (u,v): u\in U_2, v\in U_1^c\rbrace.$$
\item  Fix $0<\delta <1-\alpha$ with  $100\delta M_1 M_2 < \alpha$. 
\item For $k$ sufficiently large, we   have for every $x\in \Sigma$ that 
\begin{enumerate}
	\item $m_x^k(U_4)>\alpha - \delta ;$
	\item $m_x^k(U_0)<\alpha + \delta. $
\end{enumerate}
\item For the remainder,  fix $k$ sufficiently large so that all  estimates above (including \eqref{eq:kappa}) hold.  
\item Given our $k$ fixed above  define $\rho\colon \Sigma\to [0,1)$ so that $$m_x^k(B(q, \rho(x)))= \alpha + \delta.$$
\end{enumerate}
Observe that as $m_x^k$ is assumed to have  no atoms and as the measures $m_x^k$ vary continuously in $x$, the function $\rho$ is continuous.  We write $$U_x:= B(q, \rho(x)).$$
Note that the choices above ensure that 
$U_0\subset U_x$.

\subsection{Constructing  finite energy  families of symmetric self-couplings} 
For the remainder of this section we work exclusively with the $k$ fixed above.  
We will work primarily with the family of measures $\{\restrict {m^k_x}{U_x}\}$. Recall  that the measure $\restrict{m^k_x}{U_x}$ is defined for {every} $x\in \Sigma$.   Moreover, the dependence on $x$ is continuous. Below, we will define a number of families of symmetric self-couplings $\{\xi_x\}_{x\in \Sigma}$ of $\{\restrict{m^k_x}{U_x}\}_{x\in \Sigma}$.  For every such family $\{\xi_x\}_{x\in \Sigma}$, the measure $\xi_x$ will be defined for  every $x\in \Sigma$. We start constructing a family of symmetric self-couplings of $\{\restrict{m^k_x}{U_x}\}_{x\in \Sigma}$ with finite energy.

From the continuity and non-atomicity of the conditional measures $m_x^k$ we obtain 
\begin{claim}
There is an $r>0$ so that for every $x\in \Sigma$ and $u\in \Pone$
$$m_x^k(B(u,2r))<\frac{ \alpha + \delta}{10}.$$
\end{claim}
Using the above claim we have the following lemma. 
\begin{lemma}
There exists a family of symmetric self-couplings $\{\xi_x\}_{x\in \Sigma}$ of $\{\restrict{m^k_x}{U_x}\}_{x\in \Sigma}$ with finite energy.  
\end{lemma}
\begin{proof}
Let $\{\xi_x\}_{x\in \Sigma}$ be any family of symmetric self-couplings of $\{\restrict{m^k_x}{U_x}\}_{x\in \Sigma}$.  Let $\{v_i\}_{i= 1,\dots, m}\subset \Pone $ be such that $\Pone=\bigcup_{i=1}^m B(v_i, r)$.  
Define $\xi^1_x$ by
$$\xi^1_x:= \xi_x - \restrict{\xi_x}{B(v_1, r)\times B(v_1, r)} -\theta_x \restrict{\xi_x}{B(v_1, 2r)^c\times B(v_1, 2r)^c} + \zeta_x +  \iota _* \zeta_x $$

where 
$$\theta_x=  \dfrac{ {\xi_x}\left({B(v_1, r)\times B(v_1, r)}\right)}{{\xi_x}\left({B(v_1, 2r)^c\times B(v_1, 2r)^c}\right)}$$
and 
$$\zeta_x
:= \dfrac{1}{{\xi_x}\left({(B(v_1, r))^2}\right)} \left((\pi_1)_*\restrict{\xi_x}{B(v_1, r)^2}  \times 
(\pi_1)_*\theta_x  \restrict{\xi_x}{\left(B(v_1, 2r)^c\right)^2} \right). $$ 

As
\begin{align*}
\xi_x(B(v_1, r)\times B(v_1, r)) &\le 
\xi_x(B(v_1, 2r)\times B(v_1, 2r)) \\
& = \restrict{m^k_x}{U_x}(B(v_1, 2r)) - \xi_x(B(v_1, 2r)\times B(v_1, 2r)^c) \\
& <  \restrict{m^k_x}{U_x}(B(v_1, 2r)^c)- \xi_x( B(v_1, 2r)^c \times B(v_1, 2r)) \\
& = \xi_x( B(v_1,2 r)^c \times B(v_1, 2r)^c) \end{align*}
we have that  $\theta_x<1$ hence $\xi^1_x$ is a (positive) measure.  
Moreover, $\xi^1_x$ is clearly symmetric and we check that $(\pi_1)_*\xi^1_x= \restrict{m^k_x}{U_x}$.  

The family $\{\xi^1_x\}_{x\in \Sigma}$ depends measurably on $x$ and  satisfies $\xi^1_x(B(v_1, r) \times B(v_1, r)) = 0$ and 
$\xi^1_x(B(v_i, r) \times B(v_i, r)) \le \xi_x(B(v_i, r) \times B(v_i, r))$ for all $1\le i \le m$.  
Iterating the above construction yields a measurable family of symmetric self-couplings  $\{\xi^m_x\}_{x\in \Sigma}$ of $\{\restrict{m^k_x}{U_x}\}_{x\in \Sigma}$, defined for every $x\in \Sigma$, with 
$$\xi^\ell_x(B(v_i, r) \times B(v_i, r)) = 0 $$ for every $1\le i\le m$. Taking $0<r_0<r$ to be the Lebesgue number of the cover $\{ B(v_i, r)\}^l_{i=1}$ we have $$\xi^m_x(B(u, r_0) \times B(u,r_0))=0$$ for every $u\in \Pone$ and $x\in \Sigma$.  Then $\int \phi \ d \xi^m _x \le -\log (r_0)$ for every $x\in \Sigma$.  
\end{proof}

We introduce the first of many modifications we perform on our families of symmetric self-couplings.  
\begin{lemma}\label{lemma: coupling gives zero measure to U2 complementar big atoms}
Let $\lbrace \xi_x\rbrace _{x\in \Sigma} $ be a family of symmetric self-couplings of $\{\restrict{m^k_x}{U_x}\}_{x\in \Sigma}$ with finite energy. Then there exists a family of symmetric self-couplings $\{\dot{\xi}_x\}_{x\in \Sigma}$ of $\{\restrict{m^k_x}{U_x}\}_{x\in \Sigma}$  with $\dot{\xi}_x (U_2 ^c\times U_2 ^c) =0$ and such that for every $x\in \Sigma$
\begin{displaymath}
\int \varphi \ d\dot{\xi}_x \leq \int \varphi \ d\xi _x + 4\delta M_2.
\end{displaymath} 
\end{lemma}
\begin{proof}
Let $\nu _x:=(\pi_1)_*\restrict{\xi _{x}}{ U^c_2\times U^c_2}$   and $\eta _x= (\pi_1)_*\restrict{\xi _{x}}{ U_3\times U_3}$. Define
\begin{displaymath}
\dot{\xi}_x :=\xi _x -\restrict{\xi _{x}}{  U^c_2\times U^c_2} - \dfrac{\| \nu _x \|}{\| \eta _x \|}\restrict{ \xi _{x}}{ U_3\times U_3} + \dfrac{1}{\| \eta _x \|}\left( \nu _x \times \eta _x + \eta _x \times \nu _x \right).
\end{displaymath}

We have that $\dot{\xi}_x (U_2 ^c\times U_2 ^c) =0$,  $(\pi_j)_* (\dot{\xi}_x) = \xi_x$ for both $j= \{1,2\}$, and  that 
$\dot{\xi}_x$ is  symmetric.  
Moreover, we have 
 $$\| \eta _x \|\ge \|\xi_x\|  -  2 \xi_x(( U_x\sm U_3)\times \Pone) = 
 ( \alpha +\delta )-4\delta.$$   
and   
$$\| \nu _x \| = \|\restrict { \xi _{x}}{ U^c_2\times U^c_2} \| \le \xi _x(U_2^c\times \Pone) \le 2\delta.$$
It follows for every $x$ that $\| \nu _x \| \le \| \eta _x \|$ and hence $\dot{\xi}_x$ is a (positive) measure.

Note that $\nu_x\times \eta_x$ is supported on $U_2^c \times U_3$.  It follows that 
\begin{displaymath}
\begin{split}
\int \varphi \ d\dot{\xi} _x & \leq \int \varphi  \ d\xi  _x + \dfrac{1}{\| \eta  _x\| }\int \varphi \ d(\nu  _x \times \eta  _x + \eta  _x \times \nu  _x)\\
& \leq \int \varphi \ d\xi  _x + \dfrac{2}{\| \eta  _x\| }\int \varphi  \ d(\nu  _x \times \eta  _x )\\
& \leq \int \varphi \ d\xi  _x + \dfrac{2}{\| \eta  _x\| }\int M_2 \ d(\nu  _x \times \eta  _x )\\
& \leq \int \varphi\  d\xi  _x + \dfrac{2M_2\| \eta  _x\| \| \nu  _x\|}{\| \eta  _x\| }.
\end{split}
\end{displaymath}

As $\| \nu  _x\| \leq 2\delta$ the claim follows. 
\end{proof}

\subsection{Key proposition}
We are now in position to state the key proposition that establishes the contradiction to \eqref{eq:mainassump} in the case when the measures $m^k_x$ are non-atomic. To prove it we exploit the fact that $q$ is an expanding point for the projective cocycle $\P A^N_k$ (recall \eqref{eq: estimative 2}) and the invariance of $m_k$ with respect to $F_{A_k}$ (recall Remark \ref{remark: invariant measure}).
\begin{proposition}\label{prop:new self-coupling big atoms}\label{prop:key}
Let $\lbrace \xi _x\rbrace _{x\in \Sigma} $ be a family of symmetric self-couplings of $\{\restrict{m^k_x}{U_x}\}_{x\in \Sigma}$ with finite energy. Then, there exists a family of symmetric self-couplings $\lbrace \ddot{\xi} _x\rbrace _{x\in \Sigma} $ of $\{\restrict{m^k_x}{U_x}\}_{x\in \Sigma}$ such that
\begin{displaymath}
\int \int \varphi\  d\ddot{\xi}_x d\mu_k(x)  \leq \int \int \varphi \ d\xi _x d\mu_k(x) -\alpha. 
\end{displaymath} 
\end{proposition}

As $\phi$ is a non-negative function, by recursive applications of Proposition \ref {prop:new self-coupling big atoms} we arrive at a contradiction.

To start the proof of Proposition \ref{prop:new self-coupling big atoms}, given $\{\xi_x\}_{x\in \Sigma}$ let $\{\dot \xi_x\}_{x\in \Sigma}$ be the family of symmetric self-couplings  constructed in Lemma \ref{lemma: coupling gives zero measure to U2 complementar big atoms}.  
For each $x\in \Sigma $ define 
\begin{equation}\label{eq: almost self-coupling small atoms}
\hat{\xi}  _x = \sum _{y\in f^{-N}(x)} \dfrac{1}{J_{\mu_k}f^N(y)} \left( \P A^N_k(y) \times \P A^N_k(y)\right) _{\ast} \dot{\xi} _y.
\end{equation}
The restriction of $\hat{\xi}_x$ to $U_x \times U_x$ is not necessarily a self-coupling of $\restrict{m^k_{x}}{ U_x}$. 
Write  $\eta _x $ for the projection $\eta _x:= (\pi_1)_*\left(\restrict{\hat{\xi}_{x}}{U_x \times U_x}\right)$.   
Below, we estimate the defect between $\eta_x$ and $\restrict{m^k_{x}}{ U_x}$.

Write $g(y)= \dfrac{1}{J_{\mu_k}f^{N}(y)}.$  Recall that for any $x\in \Sigma$
\begin{displaymath}
\sum _{y\in f^{-N}(x)} g(y) 
\P A^N_k(y)_{\ast}m^k_y=m^k_x.
\end{displaymath}
Define two families of measures on $\Pone$ by 
\begin{itemize}
\item $\displaystyle I_x:= {\left(\sum _{y\in f^{-N}(x)} g(y) \P A^N_k (y)_{\ast}\left( \restrict{m^k_{y }}{  U_y^c}\right) \right)}\bigg|_{U_x}$, and 
\item $O_x:= (\pi_1)_* \left(\restrict{\hat{\xi}_{x}}{U_x \times U_x^c}\right)$.
\end{itemize}
The families $\{I_x\} _{x\in \Sigma}$ and $\{O_x\} _{x\in \Sigma}$ are measurable.  

\begin{lemma}\label{lemma: comparing the projection with the real measure}
We have 
\begin{displaymath} 
\restrict{m^k_{x}}{ U_x} =\eta _x +I_x + O_x.  
\end{displaymath}
Moreover, for every $x\in \Sigma$  we have 
$\| O_x\| \leq \| I_x \| \leq 2\delta $ and $\textrm{supp}(I_x)\subset U^c_1$. 
\end{lemma}
\begin{proof}

We have 
\begin{align}
\restrict{m^k_{x}}{ U_x} 
	&= \left( \sum _{y\in f^{-N}(x)} g(y) \P A^N_k(y)_{\ast}m_y^k\right)\bigg|_{U_x}\notag \\
	&=\left( \sum _{y\in f^{-N}(x)} g(y) \P A^N_k(y)_{\ast}(\restrict{m_y^k}{U_y^c})\right)\bigg|_{U_x} \label{eq:split2}  \\ &  +
	\left( \sum _{y\in f^{-N}(x)} g(y) \P A^N_k(y)_{\ast}(\restrict{m_y^k}{ U_y})\right)\bigg|_{U_x}.  \label{eq:split}
\end{align}
The  term  \eqref{eq:split2} is precisely $I_x$.  The  term  \eqref{eq:split} is 
$$\restrict{\left((\pi_1)_*\hat \xi_x \right)}{U_x} =(\pi_1)_*(\restrict{\hat \xi_x}{U_x \times U_x}) + O_x$$ 
hence we obtain 
$$\restrict{m^k_{x}}{ U_x} =\eta _x +I_x + O_x.  $$

To bound $\| I_x\|$ note that for any measurable  set $B\subset \Pone$ we have  $I_x(B)  \le m_x^k(B)$.  Moreover, $I_x$ is supported on $$\bigcup _{y\in f^{-N}(x) } \P  A^N_k (y)(U_y^c)\subset \bigcup _{y\in f^{-N}(x) } \P  A^N_k (y)(U_0^c) \subset 
U_1 ^c.$$ 
Thus $\|I_x\| \le m_x^k (U_x\sm U_1) \le 2\delta$.  

To derive the bound on $\|O_x\|$
first note that 

\begin{equation*}
m^k_x(U_x)
=\sum _{y\in f^{-N}(x)} g(y)  \P A^N_k (y)_{\ast} \left( \restrict{m^k_{y}}{ U_y}\right) (U_x) \\
+\sum _{y\in f^{-N}(x)} g(y) \P A^N_k (y)_{\ast}\left(\restrict{ m^k_{y}}{  U_y^c}\right) (U_x).
\end{equation*}
Similalry, 
\begin{align*}
\sum _{y\in f^{-N}(x)} g(y) \P A^N_k (y)_{\ast}\left( \restrict{m^k_{y}}{ U_y}\right) (U_x) 
&+ \sum _{y\in f^{-N}(x)} g(y) \P A^N_k (y)_{\ast}\left( \restrict{m^k_{y}}{ U_y}\right) (U_x^c)\\
&= \sum _{y\in f^{-N}(x)} g(y) \P A^N_k (y)_{\ast}\left( \restrict{m^k_{y}}{ U_y}\right) (\Pone) \\
&=  \sum _{y\in f^{-N}(x)} g(y)  \left( m^k_{y}( U_y)\right) \\
& = \alpha+\delta= m^k_x(U_x)
\end{align*}
where the final equality follows from the choice of open sets $U_y$. 
Combining the above we have for any $x\in \Sigma$ that 
\begin{equation}\label{eq:banana}
\sum _{y\in f^{-N}(x)} g(y) \P A^N_k (y)_{\ast}\left( \restrict{m^k_{y}}{  U_y^c}\right) (U_x)= \sum _{y\in f^{-N}(x)} g(y) \P A^N_k (y)_{\ast}\left( \restrict{m^k_{y}}{ U_y}\right) (U_x^c).
\end{equation}

The lefthand side of \eqref{eq:banana} is $I_x$.  The righthand side of \eqref{eq:banana} is 
$$(\pi_2)_*(\restrict{\hat \xi }{\Pone \times U_x^c}).$$
Then 
$$\|I_x\| = \|(\pi_2)_*(\restrict{\hat \xi }{\Pone \times U_x^c})\| \ge 
\|(\pi_2)_*(\restrict{\hat \xi }{U_x \times U_x^c})\| = \|O_x\| 
 $$
and the lemma follows.\end{proof}

\subsection{Proof of Proposition \ref{prop:key}}
We conclude this section with the proof of Proposition \ref{prop:key}.  
\begin{proof}[Proof of Proposition \ref{prop:key}]
Recall the notation and  constructions above.  
Define measurable families of measures  
\begin{itemize}\item 
$\theta _x: = (\pi_1)_* \left(\restrict{\hat{\xi}_{x}}{  U_3\times U_3}\right) ;$ 
\item
$\lambda _x :=\left( 1-\dfrac{\|\restrict{ O_{x}}{  U_2} \|}{\| I_x\|}\right) I_x + \restrict{O^k_{x}}{ U^c_2}.$
\end{itemize}
For $x\in \Sigma$ we define a new measure on $\Pone\times \Pone$ by 
\begin{equation}\label{eq: definition of the new coupling small atoms }
\begin{split}
\ddot{\xi} _x := & \  \restrict{\hat{\xi}_{x}}{U_x\times U_x} -  \dfrac{\|\lambda  _x \|}{\| \theta  _x \|} \restrict{\hat{\xi} _{x}}{ U_3\times U_3}  \\
&+ \dfrac{1}{\| I _x \| } \left( \restrict {O _{x}}{ U_2} \times I _x + I _x \times \restrict {O _{x}}{ U_2} \right) \\
&+ \dfrac{1}{\| \theta  _x \| } \left(\lambda  _x \times \theta  _x + \theta  _x \times \lambda  _x \right).
\end{split}
\end{equation}
The family $\{\ddot{\xi}_x\}_{x\in \Sigma}$ is  measurable.  

As $
\| O _x \| \leq \| I _x \| $, we have that $\lambda  _x$ is a (positive) measure. Moreover, we have 
$$ \| \theta  _x \| \ge 
\|\restrict{\hat{\xi}_{x}}{  U_3\times \Pone} \| - 
\| \restrict{\hat{\xi}_{x}}{  U_3\times U_3^c} \|\geq m_x^k(U_3) - m_x^k(U_x\sm U_3)\ge \alpha -3\delta.$$
As 
$$\|\lambda  _x \| \leq 4\delta\le \alpha - 3 \delta \le  \| \theta  _x \|$$
we have  that $\ddot{\xi} _x$ is a (positive) measure.  Also $\ddot{\xi} _x$  is clearly symmetric.

Let $D\subset \Pone$ be a measurable set. Then,
\begin{displaymath}
\begin{split}
(\pi _{1})_* \ddot{\xi} _x(D)&=\ddot{\xi} _x(D\times \Pone)= \eta  _x(D)- \dfrac{\| \lambda  _x\|}{\| \theta  _x\| }\theta  _x(D)\\ 
&+  O _x(D\cap U_2) + \dfrac{1}{\| I _x\|}I _x(D)O _x(U_2)+ \lambda  _x(D) +\dfrac{\| \lambda  _x\|}{\| \theta  _x\| }\theta  _x(D)\\
&=\eta  _x(D)+O _x(D\cap U_2)+ \dfrac{1}{\| I _x\|}I _x(D)O _x(U_2)+ \lambda  _x(D)\\
& =\eta _x (D)+ \restrict {O _{x}}{ U_2} (D)+\dfrac{\| \restrict {O _{x}}{ U_2}\|}{\| I _x\|}I _x  (D)+ \left( 1-\dfrac{\| \restrict {O _{x}}{ U_2} \|}{\| I _x\|}\right) I _x  (D)+ \restrict{O _{x}}{  U^c_2} (D)\\
&= \eta  _x  (D)+ O _x  (D)+ I _x (D) =\restrict{m^k_{x}}{ U_x} (D)  
\end{split}
\end{displaymath}
hence the family  $\{\ddot{\xi} _x\}_{x\in \Sigma}$ is a family of  symmetric self-couplings of $\{\restrict{m^k _{x}}{ U_x}\}_{x\in \Sigma}$.

From the definition of $\ddot{\xi}_x$ we have 
\begin{equation}\label{eq: limiting the energy of the new couplin small atoms}
\begin{split}
\int \int \varphi \  d\ddot{\xi}_x d\mu_k(x) & \leq \int \int \varphi \ d\restrict{\hat{\xi}_{x}}{  U_x\times U_x} \ d\mu_k(x)\\
& +  \int\dfrac{1}{\| I _x \| } \int \varphi  \ d \left( \restrict {O _{x}}{ U_2} \times I _x + I _x \times \restrict {O _{x}}{ U_2} \right) d\mu_k(x) \\
&+ \int  \dfrac{1}{\| \theta  _x \| } \int \varphi \ d\left( \lambda  _x \times \theta  _x + \theta  _x \times \lambda  _x \right)\ d\mu_k(x).
\end{split}
\end{equation}

\begin{figure}[hbtp]
\centering
\includegraphics[scale=1]{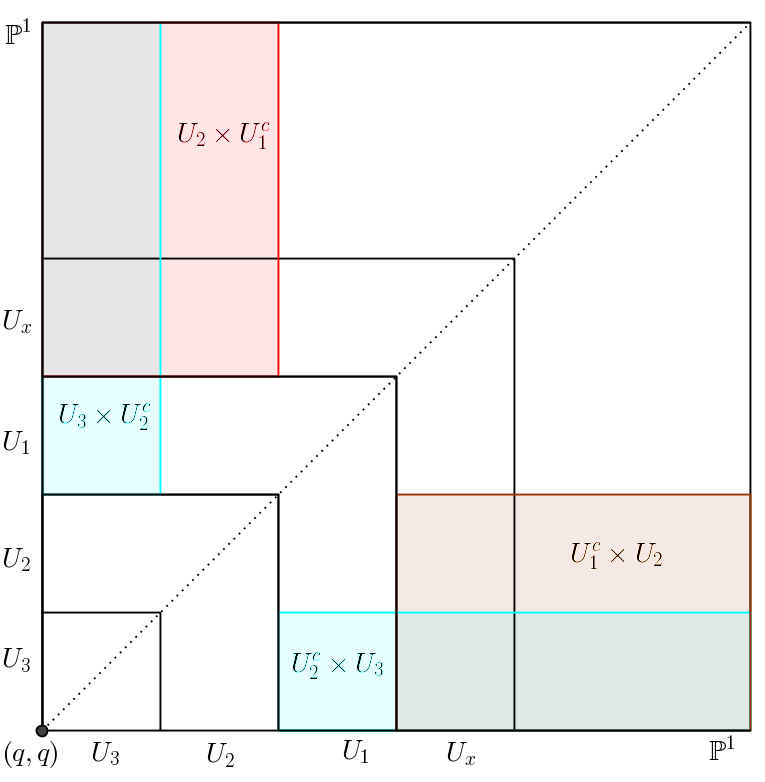}
\caption{Mass away from the diagonal. \label{figure: proj2}}
\end{figure}

Since  $\textrm{supp}(I _x) \subset U^c_1$ and $\textrm{supp}(\restrict{O _{x}}{ U_2})\subset U_2$ we have 
\begin{align*}
 \int\dfrac{1}{\| I _x \| } \int \varphi  &\ d \left( \restrict {O _{x}}{ U_2} \times I _x + I _x \times \restrict {O _{x}}{ U_2} \right) d\mu_k(x)\\
&=  \int\dfrac{2}{\| I _x \| } \int \varphi  \ d \left( \restrict {O _{x}}{ U_2} \times I _x \right) d\mu_k(x)\\
&   \leq \int \dfrac{2}{\| I _x \| } \int M_2  \ d \left( \restrict {O _{x}}{ U_2} \times I _x \right) d\mu_k(x) 
\\& \leq  \int \dfrac{2M_2 \| I _x \| O _x ( U_2)}{\| I _x \| }  \ d \mu_k(x)\leq 4\delta M_2.
\end{align*}
Similarly,  $\textrm{supp}(\theta _x)\subset U_3$ and $\textrm{supp}(\lambda _x)\subset U_2^c$ hence 
\begin{align*}
\int \dfrac{1}{\| \theta  _x \| } \int \varphi \ & d\left( \lambda  _x \times \theta  _x + \theta  _x \times \lambda  _x \right)d\mu_k(x)
\\&\leq \int  \dfrac{2}{\| \theta  _x \| } \int \varphi \ d\left( \lambda  _x \times \theta  _x \right)d\mu_k(x)\\&
\\& \leq   \int \dfrac{2}{\| \theta  _x \| }\int M_2 \ d\left( \lambda  _x \times \theta  _x \right)d\mu_k(x) 
\\& \leq  \int \dfrac{2 M_2 \| \theta  _x \| \| \lambda  _x \|}{\| \theta  _x \| }  \ d \mu_k (x) \leq 8\delta M_2
\end{align*}

Moreover, from the construction of  $\ddot{\xi} _x$ and recalling that   $\dot{\xi} _y (U_2 ^c\times U_2 ^c) =0$ for all $y $ we have 
\begin{align*}
 \int \int \varphi \ &d\restrict{\hat{\xi}_{x}}{U_x\times U_x} \ d\mu_k(x)
 \le  \int \int \varphi \ d\hat{\xi}_x \ d\mu_k(x) \\
&=\int  \sum _{y\in f^{-N}(x)} g(y)\left( \int _{U_y\times U_y} \varphi \left(\P A^N_k(y)(u),\P A^N_k(y)(v)\right)  d\dot{\xi} _y \right) d\mu_k(x)\\
&\le \int  \sum _{y\in f^{-N}(x)} g(y) \left(
\int _{U_1\times U_1} \varphi \left(\P A^N_k(y)(u),\P A^N_k(y)(v)\right) d\dot{\xi} _y \right)  d\mu_k(x)\\
&+\int  \sum _{y\in f^{-N}(x)} g(y) \left( \int _{(U_2\times U_1^c) \cup (U_1^c \times U_2)}  \varphi \left(\P A^N_k(y)(u),\P A^N_k(y)(v)\right)   d\dot{\xi} _y \right)   d\mu_k(x).
\end{align*}

  Recalling the choice of $M_1, M_2 $ and $\delta $ we have 
\begin{align*}
\int  \sum _{y\in f^{-N}(x)}g(y) & \left( \int _{(U_2\times U_1^c) \cup (U_1^c\times U_2)}  \varphi \left(\P A^N_k(y)(u),\P A^N_k(y)(v)\right)\ d\dot{\xi} _y \right) d\mu_k(x)\\
&\le 2 \int  \sum _{y\in f^{-N}(x)}g(y)  \left( \int _{U_2\times U_1^c} M_1 \varphi (u, v) \ d\dot{\xi} _y \right)  d\mu_k(x)
\\ &\leq 2M_1M_2  \int  \dot{\xi} _x(U_1^c\times U_2) \ d\mu_k(x) \\ &\leq 2M_1M_2  \int  \restrict{m^k_x}{U_x} (U_1^c) \ d\mu_k(x) \leq 4 \delta M_1 M_2.
\end{align*}

On the other hand, it follows from \eqref{eq: estimative 2} and Lemma \ref{lemma: coupling gives zero measure to U2 complementar big atoms} that

\begin{displaymath}
\begin{split}
\int  \sum _{y\in f^{-N}(x)} & g(y) \left(
\int _{U_1\times U_1}  \varphi \left(\P A^N_k(y)(u),\P A^N_k(y)(v)\right)   d\dot{\xi} _y \right)  d\mu_k(x)\\
&\leq \int  \int _{U_1\times U_1} \varphi (u,v)\  d\dot{\xi} _x d\mu_k(x) -\int  \dot{\xi} _x (U_1\times U_1)\kappa (x) \ d\mu_k(x)   + \alpha\\
&\leq \int  \int _{\Pone \times \Pone} \varphi (u,v)\  d\dot{\xi} _x d\mu_k(x) -\int  \dot{\xi} _x (U_1\times U_1)\kappa (x) \ d\mu_k(x)   + \alpha \\
& \leq \int  \int _{\Pone \times \Pone } \varphi (u,v)\  d\xi  _x d\mu_k(x) + 4\delta M_2 -\int  \dot{\xi} _x (U_1\times U_1)\kappa (x) \ d\mu_k(x)  + \alpha.
\end{split}
\end{displaymath}

Note that for any family of symmetric self-couplings $\{\dot \xi_x\}_{x\in \Sigma}$ of $\{ \restrict {m^k_x}{U_x}\}_{x\in \Sigma}$ we have $$\alpha - 3\delta \le  \dot\xi_x(U_1\times \Pone) - \dot \xi_x(U_1\times(U_x\sm U_1))= 
\dot \xi_x(U_1\times U_1) \le \dot \xi_x(U_1\times \Pone) \le \alpha + \delta.$$

Writing  $\kappa ^+ (x)$ and $\kappa ^-(x)$,  respectively,  for  the positive and negative parts of $\kappa (x)$ we have that
\begin{displaymath}
\begin{split}
&\int  \dot{\xi} _x (U_1\times U_1)\kappa (x)\ d\mu_k(x)\\
& \geq (\alpha-3\delta)\int  \kappa ^ +(x)d\mu_k(x) -(\alpha+ \delta )\int  \kappa ^- (x)\ d\mu_k(x)\\
&= (\alpha-3\delta)\int  \kappa (x)\ d\mu_k(x) -4\delta \int  \kappa ^- (x)\ d\mu_k(x)\\
& \geq (\alpha-3\delta)4 -4\delta M_1\geq 4\alpha- 16\delta M_1> 3\alpha.
\end{split}
\end{displaymath}

We therefore have 
\begin{align*}
\int \int \varphi \ & d\ddot{\xi}_x d\mu_k(x) \\&
\leq \int  \int  \varphi \  d\xi _x d\mu_k(x) + 4\delta M_2 - 3\alpha  + 4\delta M_1 M_2 + 4\delta M_2 + 8\delta M_2 + \alpha\\
&
\leq \int  \int  \varphi  \  d\xi _x d\mu_k(x)  -\alpha . \qedhere
\end{align*}
\end{proof}
This completes the proof of Proposition \ref{prop:key}.  Combined with the results of Section \ref{sec:atomic case} this completes the proof of Theorem \ref{mainthm}.


\bibliography{bibliography}

\end{document}